\documentclass[11pt,twoside,reqno,centertags,english]{amsart}
\usepackage{kpfonts}
\usepackage[utf8]{inputenc}
\usepackage{babel}
\usepackage{fancyhdr}
\usepackage{geometry}
\usepackage{hyperref}
\hypersetup{
colorlinks,
citecolor=blue,
filecolor=blue,
linkcolor=blue,
urlcolor=blue
}

\usepackage{amsmath, amsfonts, amssymb}
\usepackage{amsthm}
\newtheorem{theorem}{Theorem}[section]
\newtheorem{lemma}[theorem]{Lemma}

\newtheorem{corollary}[theorem]{Corollary}

\theoremstyle{remark}
\newtheorem{remark}[theorem]{Remark}
\theoremstyle{definition}
\newtheorem{definition}[theorem]{Definition}

\DeclareMathOperator{\grad}{grad}
\DeclareMathOperator*{\esslimsup}{ess\;lim\;sup}
\DeclareMathOperator*{\esssup}{ess\;sup}

\DeclareMathOperator{\Div}{div}
\newcommand{\energy}{\mathcal W}
\newcommand{\entropy}{\mathcal E}
\newcommand{\entropyproduction}{D\mathcal E}
\newcommand {\R} {\mathbb{R}}
\newcommand{\cd}{\,d}
\DeclareMathOperator{\dive}{div}
\newcommand{\problem}{\eqref{eq:p1}--\eqref{eq:p4}}
\newcommand{\Pp}{{\mathcal{P}}}
\newcommand{\Mm}{{\mathcal{M}^+}}

\numberwithin{equation}{section}

\usepackage{etoolbox} 

\date{}

\title[Spherical gradient flows]{Spherical Hellinger-Kantorovich gradient flows}

\author[S.~Kondratyev]{Stanislav Kondratyev}
\address[S.~Kondratyev]{CMUC, Department of
Mathematics, University of Coimbra, 3001-501 Coimbra, Portugal}{}
\email{kondratyev@mat.uc.pt}

\author[D.~Vorotnikov]{Dmitry Vorotnikov}
\address[D.~Vorotnikov]{CMUC, Department of
Mathematics, University of Coimbra, 3001-501 Coimbra, Portugal}{}
\email{mitvorot@mat.uc.pt}

\begin{document}

\begin{abstract} We study nonlinear degenerate parabolic equations of Fokker-Planck type which can be viewed as gradient flows with respect to the recently introduced spherical Hellinger-Kantorovich distance. The driving entropy is not assumed to be geodesically convex. We prove solvability of the problem and the entropy-entropy production inequality, which implies exponential convergence to the equilibrium. As a corollary, we obtain some related results for the Wasserstein gradient flows. We also deduce transportation inequalities in the spirit of Talagrand, Otto and Villani for the spherical and conic Hellinger-Kantorovich distances.

\end{abstract}
\maketitle

\vspace{10pt}
\textbf{MSC [2010] 26D10, 35Q84, 49Q20, 58B20}

\section{Introduction}

Unbalanced optimal transport \cite{LMS18,KMV16A,CP18,LMS16,CPSV18,Rez15} is a recent variant of the Monge-Kantorovich transport which is relevant in the situations lacking the conservation of the total mass, such as processes involving reaction. Important objects in the field are the conic Hellinger-Kanto\-rovich distance (also known as the Wasserstein-Fisher-Rao distance) on the set of Radon measures and the spherical Hellinger-Kantoro\-vich distance on the set of probability measures, see Section  \ref{sec:tala} below for the definitions and references. 

On both the conic and spherical Hellinger-Kantoro\-vich spaces, some Otto calculus \cite{otto01,Vil08} can be developed \cite{KMV16A,BV18}, and it is easy to formally define the gradient flows. This paper considers the spherical gradient flows. 

Our basic setting is as follows. Let $\Omega$ be either an open connected bounded domain in $\mathbb 
R^d$ with sufficiently smooth boundary or a flat torus $\mathbb T^d$.  Fix functions $E \in C(\overline 
\Omega \times [0, \infty))$, $f \in C^1(\overline \Omega \times (0, 
+\infty))$, and a probability density $m \in C(\overline \Omega)$ satisfying
\begin{gather}
E(x, u) \ge 0, \quad (x, u) \in \overline \Omega \times [0, \infty);
\label{eq:eep-Ei}
\\
m(x) > 0, \quad x \in \overline \Omega;
\label{eq:eep-mi}
\\
E(x, m(x)) = 0, \quad x \in \Omega;
\label{eq:eep-Emi}
\\
E_u(x, u) = -f(x, u), \quad (x, u) \in \Omega \times (0, +\infty);
\label{eq:eep-Efi}
\\
f_u(x, u) < 0, \quad (x, u) \in \overline \Omega \times (0, +\infty)
\label{eq:eep-fi}
.
\end{gather}
Here we opted to fix $E$, $f$, $m$ satisfying some hypotheses, but it is 
possible to state all the assumptions in terms of $f$ only, and then 
reconstruct $E$ and $m$ in a relevant way, see Section \ref{s:gf}.  Some 
examples are presented in Section~\ref{ss:examples}.

The function
\begin{equation} \label{eq:entri}
\entropy (u) = \int_\Omega E(x, u(x)) \,dx
.
\end{equation} will be called the \emph{relative entropy}. 

We are interested in the formal gradient flow \begin{equation}
\partial_t u=-\grad \mathcal{E}(u),
\label{eq:gradf}
\end{equation} where the gradient is taken w.r.t. the spherical Hellinger-Kantorovich structure on the set of probability measures on $\Omega$. More specifically, we study the problem
\begin{align}
\partial_t u & = -\Div (u \nabla f) + u\left( f - \int_\Omega uf\, dx \right), 
& (x, t) & \in \Omega \times (0, \infty),
\label{eq:p1o}
\\
u \frac{\partial f}{\partial \nu} & = 0, & (x, t) & \in \partial \Omega \times 
(0, \infty),
\label{eq:p2o}
\\
u & = u^0, & (x, t) & \in \Omega \times {0},
\label{eq:p3o}
\\
u & \ge 0, \ \int_\Omega u \,dx = 1,
& (x, t) & \in \Omega \times (0, \infty)
.
\label{eq:p4o}
\end{align}
We refer to Remark \ref{slop} concerning the relation between \eqref{eq:gradf} and this problem. The model \eqref{eq:p1o}--\eqref{eq:p4o} can be viewed as a reactive nonlinear equation of Fokker-Planck type, in the spirit of \cite{F04}, with conservation of mass. Reaction-diffusion problems with conservation of mass were studied in \cite{QS07,HY95,RS92,Sou98,alr16,HMTW16,KHL17}, see also the references therein. On the other hand, after a change of variables, our problem fits into the framework of fitness-driven models of population dynamics, and might be applicable to some human societies. In Remark \ref{rem:hum} we discuss this issue in detail. 

\begin{remark}\label{slop} The right-hand sides of \eqref{eq:gradf} and 
\eqref{eq:p1o} formally coincide when $\Omega$ is a torus or is 
convex.  Indeed, the gradient under these assumptions was calculated 
in~\cite{LM17,BV18}:
\begin{equation*}
\grad_{HKS} \mathcal{E}(u) =-\Div \left(u \nabla \frac{\delta \mathcal 
        E}{\delta u}\right) + u\left( \frac{\delta \mathcal E}{\delta u} - 
        \int_\Omega u \frac{\delta \mathcal E}{\delta u}\, dx \right)
.
\end{equation*}
In the case of non-convex $\Omega$ we will still refer to  
\eqref{eq:p1o}--\eqref{eq:p4o} as to a gradient flow, although this is sloppy. 
\end{remark}

\begin{remark} For the metric gradient flows like \eqref{eq:gradf}, the geodesic convexity of the driving entropy functional (or at least semi-convexity, i.e., $\lambda$-convexity with a negative constant $\lambda$) makes a difference \cite{otto01,AGS06,Vil03,Vil08}. 
The presence of convexity allows one to apply minimizing movement schemes 
\cite{AGS06,JKO} to construct solutions to the gradient flow.  Moreover, 
$\lambda$-convexity with $\lambda$ strictly positive enables the Bakry-Emery 
procedure \cite{BE85} which usually yields the exponential convergence of the relative 
entropy to zero. Minimizing movement schemes for conic Hellinger-Kantorovich 
gradient flows of geodesically convex functionals and for related 
reaction-diffusion equations were suggested in \cite{MG16,GLM17}.  

Under our assumptions, the entropy, generally speaking, possesses neither 
geodesic convexity nor semi-convexity with respect to either the spherical or 
conic Hel\-linger-Kantorovich structure, or even to the classical Wasserstein 
one, cf. \cite{KV17,KMV16A}. \end{remark}

\begin{remark}\label{rem:hum} The fitness-driven models \cite{mc90,cosner05,cos13,T18} of population dynamics assume that the 
dispersal strategy is determined by a local intrinsic characteristic of 
organisms called \emph{fitness}. The 
fitness manifests itself as a growth rate, and simultaneously affects the 
dispersal as the species move along its gradient towards the most favorable 
environment. In terms of the PDEs, this can be expressed \cite{KV17} in the following manner: \begin{align}
\partial_t U & = -\Div (U \nabla F) + UF, 
& (x, t) & \in \Omega \times (0, \infty),
\label{eq:p1h}
\\
U \frac{\partial F}{\partial \nu} & = 0, & (x, t) & \in \partial \Omega \times 
(0, \infty).
\label{eq:p2h}
\\
U & = U^0, & (x, t) & \in \Omega \times {0}.
\label{eq:p3h}
\end{align}
Here $U(x,t)$ is the nonnegative density of individuals, and 
$F$ is the fitness which depends on $x$ and $U$ in a certain way. Namely, we 
assume that
\begin{equation} \label{Fvar} F(x,t)=f\left(x,\frac {U(x,t)} {\int_\Omega 
U(\xi,t)\, \cd \xi} \right).
\end{equation}
The direct dependence on $x$ expresses the spatial inhomogeneity of 
the resources. The dependence on the normalized population density (in 
contrast with \cite{mc90,cosner05,cos13,CW13,KV17} and the references 
therein, where the fitness depends on the density $U$ itself) models 
the phenomenon that the individuals compare the quality of their life 
with the ones of the other members of the society, and their fitness 
is determined by their relative success in comparison with the others. 
This model seems to be specifically relevant for those human societies 
where the population growth (which depends on various factors 
including  fertility, ability of children to survive, longevity etc.) 
is an increasing function of the quality of life. The problem 
\eqref{eq:p1h}--\eqref{eq:p3h} resembles a conic Hellinger-Kantorovich 
gradient flow, cf. \cite{KV17}, but this guess is wrong. The reason is 
that \eqref{Fvar} is not an $L^2$ variation of any functional. Setting 
$$M:={\int_\Omega U\, \cd x},\ u:=\frac {U} {M},\ M^0:={\int_\Omega 
U^0\, \cd x},\ u^0:=\frac {U^0} {M^0},$$ we recast \eqref{eq:p1h}, 
\eqref{eq:p2h} in the form \begin{align}
\partial_t u & = -\Div (u \nabla f) + u\left( f - \frac {d (\log M)} {dt}\right), 
& (x, t) & \in \Omega \times (0, \infty),
\label{eq:p1r}
\\
u \frac{\partial f}{\partial \nu} & = 0, & (x, t) & \in \partial \Omega \times 
(0, \infty).
\label{eq:p2r}
\\
u & = u^0, & (x, t) & \in \Omega \times {0},
\label{eq:p3r}
\\
u & \ge 0, \ \int_\Omega u \,dx = 1,
& (x, t) & \in \Omega \times (0, \infty)
.
\label{eq:p4r}
\end{align} Since $u(t)$ is a probability distribution, we at least formally infer that \begin{equation}\frac {d (\log M)} {dt}=\int_\Omega uf\, dx,\label{mass}\end{equation} arriving at \eqref{eq:p1o}--\eqref{eq:p4o}. On the other hand, given $U^0$ (and thus $u^0$ and $M^0$) and a solution $u$ to \eqref{eq:p1o}--\eqref{eq:p4o}, we can recover the mass $M(t)$ from \eqref{mass}, and $U=Mu$ solves \eqref{eq:p1h}--
\eqref{eq:p3h}. \end{remark}

In what follows, $d_{HK}$, $d_{HKS}$, and $W_2$ stand for the Hellinger-Kantorovich distance (which will be also referred to as the conic distance), spherical Hellinger-Kantorovich distance and the quadratic Wasserstein distance. Observe that \begin{equation} \label{compaint} d_{HK}\leq d_{HKS}\leq W_2\end{equation} for probability measures (see Section \ref{sec:tala} below), although $d_{HK}$ is of course defined for Radon measures of any mass.

In this paper, we prove solvability (Section \ref{s:gf}) and the entropy-entropy production inequality (Section \ref{s:ineq}) for the spherical Hellinger-Kantorovich gradient flow \eqref{eq:gradf}, and derive a related transportation inequality in the spirit of Talagrand, Otto and Villani. We also deduce some results of this kind for the Wasserstein and the conic Hellinger-Kantorovich gradient flows. As was already anticipated, we do not assume geodesic convexity of the driving entropies of the gradient flows. In order to better illustrate our results and compare them with the existing ones, let us formally write down the conceivable inequalities. 

The following four inequalities are expected to hold under the assumption $\int_\Omega u=1$:
\begin{gather}
\entropy(u) \lesssim \int_\Omega u |\nabla f|^2,
\label{eq:goalfa1}
\\
\entropy(u) \lesssim \int_\Omega u\left( f - \int_\Omega uf \right)^2 + 
\int_\Omega u |\nabla f|^2,
\label{eq:goalfa2}
\\
W_2^2(u,m) \lesssim \entropy(u),
\label{eq:goalfa4}
\\
d_{HKS}^2(u,m) \lesssim \entropy(u).
\label{eq:goalfa5}
\end{gather}
The next two inequalities do not require that $\int_\Omega u=1$:
\begin{gather}
\entropy(u) \lesssim \int_\Omega uf^2 + \int_\Omega u |\nabla f|^2,
\label{eq:goalfa3}
\\
d_{HK}^2(u,m) \lesssim \entropy(u).
\label{eq:goalfa6}
\end{gather}

Inequalities \eqref{eq:goalfa1},\eqref{eq:goalfa2},
\eqref{eq:goalfa3} are the entropy-entropy production inequalities for the Wasserstein, spherical Hellinger-Kantorovich and conic Hellinger-Kantorovich gradient flows, respectively. Inequalities \eqref{eq:goalfa4},\eqref{eq:goalfa5},
\eqref{eq:goalfa6} are the transportation (Talagrand) inequalities in those spaces. Note that \eqref{eq:goalfa1} implies \eqref{eq:goalfa2}, and \eqref{eq:goalfa2} yields \eqref{eq:goalfa3} since $$\int_\Omega u\left( f - \int_\Omega uf \right)^2=\int_\Omega uf^2-\left(\int_\Omega uf\right)^2.$$ However, the last implication is only valid for probability distributions $u$, whereas \eqref{eq:goalfa3} would not be a consequence of \eqref{eq:goalfa2} for $u$ of arbitrary mass. These three inequalities can be used to derive exponential convergence to the equilibrium $m$ for the corresponding gradient flows, see \cite{Vil03,Vil08,KV17} as well as  Theorems  \ref{th:convergence} and \ref{th:convergencew} below.

Due to \eqref{compaint}, inequality \eqref{eq:goalfa4} implies \eqref{eq:goalfa5}, and \eqref{eq:goalfa5} yields \eqref{eq:goalfa6} for probability distributions. Generally speaking, \eqref{eq:goalfa6} is not a corollary of \eqref{eq:goalfa5} (cf. Remark \ref{r:cons} below).

Inequality \eqref{eq:goalfa1} was proved in \cite{CJM01} via the Bakry-Emery 
approach provided the entropy is strictly geodesically convex w.r.t. the 
Wasserstein structure (displacement convex). It may be viewed as a generalized 
log-Sobolev inequality. The classical log-Sobolev corresponds to the case $f=- 
\log u$. Inequality \eqref{eq:goalfa2} will be proved in Section \ref{s:ineq} 
without assuming any kind of geodesic convexity. This inequality can be used 
to derive \eqref{eq:goalfa1} for geodesically non-convex entropies (see 
Section \ref{s:logs}) provided $u$ satisfies the Poincar\'e inequality (this 
is true for instance when $u$ is a Muckenhoupt weight \cite{FKS82}). 
Inequality \eqref{eq:goalfa3} was established in \cite{KV17} and will be used 
in the proof of \eqref{eq:goalfa2}.  Inequality \eqref{eq:goalfa4} was proved 
in \cite{T96,OV00,CMcV03,CGH04} (mainly for the case $\Omega=\R^d$) for 
strictly displacement convex entropies. Inequalities \eqref{eq:goalfa5} and 
\eqref{eq:goalfa6} will be proved in Section \ref{sec:tala}, again without 
assuming any geodesic convexity.

\section{Spherical inequality} \label{s:ineq}

Let $\Omega$ be an open connected bounded domain in $\mathbb 
R^d$ with sufficiently smooth boundary. The results of the section remain valid for the torus $\Omega=\mathbb T^d$.  Throughout the section, we will work with functions $E \in C(\overline 
\Omega \times [0, \infty))$, $f \in C^1(\overline \Omega \times (0, 
+\infty))$, and a probability density $m \in C(\overline \Omega)$ satisfying
\begin{gather}
E(x, u) \ge 0, \quad (x, u) \in \overline \Omega \times [0, \infty);
\label{eq:eep-E}
\\
m(x) > 0, \quad x \in \overline \Omega;
\label{eq:eep-m}
\\
E(x, m(x)) = 0, \quad x \in \Omega;
\label{eq:eep-Em}
\\
E_u(x, u) = -f(x, u), \quad (x, u) \in \Omega \times (0, +\infty);
\label{eq:eep-Ef}
\\
f_u(x, u) < 0, \quad (x, u) \in \overline \Omega \times (0, +\infty)
\label{eq:eep-f}
.
\end{gather}

In what follows, bare $f$ stands for $f(x, 
u(x))$, where $u \in U$ is given; likewise, $\nabla f$ stands for the full gradient 
of $f(x, u(x))$ with respect to~$x$.

The following theorem states the main result.

\begin{theorem}
\label{th:ineq-main}
Assume \eqref{eq:eep-E}--\eqref{eq:eep-f}. Let $U$ be a uniformly integrable 
set of smooth probability measures on $\overline\Omega$. Then, for all 
$u \in U$ and $a \in \mathbb R$,
\begin{equation}
\label{eq:goalfa}
\int_\Omega E(x, u(x)) \, dx
\le C\left[\int_\Omega u(x)(f(x, u(x))-a)^2 \, dx
+ \int_\Omega u(x)|\nabla f(x, u(x))|^2 \, dx\right],
\end{equation}
where the constant $C$ may depend on $U$ but is independent of $u$ and $a$.
\end{theorem}

By approximation, this theorem can be extended to non-smooth functions: see, 
for instance, our Theorem~\ref{th:eep}.

Our strategy of the proof of Theorem \ref{th:ineq-main} consists in proving the inequality
\begin{equation}
\label{eq:_goal}
\int_\Omega u(f - a)^2 \, dx
+ \int_\Omega u|\nabla f|^2 \, dx
\ge \varkappa a^2
\end{equation}
with a constant $\varkappa > 0$ independent of~$u$ ranging over a uniformly 
integrable set~$U$.  Indeed, by \cite[Theorem 2.9]{KV17}, we have the 
inequality
\begin{equation*}
\int_\Omega E \, dx
\le
C_1
\int_\Omega u(f^2 + |\nabla f|^2) \, dx
\end{equation*}
(we can apply the theorem because uniform integrability ensures that no 
sequence in~$U$ converges to~$0$ in measure).  Setting
\begin{equation*}
\bar f = \int uf \, dx
\end{equation*}
and recalling that~$u$ is a probability measure, we see that
\begin{equation*}
\int_\Omega uf^2 \, dx = \int_\Omega u(f - \bar f)^2 \, dx + \bar f^2
,
\end{equation*}
so if we had~\eqref{eq:_goal}, we would apply it for $a = \bar f$ obtaining
\begin{equation*}
\int_\Omega uf^2 \,dx \le (1 + \varkappa^{-1}) \int_\Omega u(f - \bar f)^2 \, 
dx+\varkappa^{-1}\int_\Omega u|\nabla f|^2 \, dx,
\end{equation*}
and thus,
\begin{equation*}
\int_\Omega E \, dx
\le
C \left[
\int_\Omega u(f - \bar f)^2 \, dx
+ \int_\Omega u|\nabla f|^2 \, dx
\right]
.
\end{equation*}
This particular case of~\eqref{eq:goalfa} actually implies~\eqref{eq:goalfa}, 
as
\begin{equation*}
\int_\Omega u(f - \bar f)^2 \,dx
= \min_{a \in \mathbb R}
\int_\Omega u(f - a)^2 \,dx
,
\end{equation*}
which is a consequence of the following instance of the Pythagorean Theorem in 
$L^2(du)$:
\begin{equation*}
\int_\Omega u(f - a)^2 \, dx = \int_\Omega u(f - \bar f)^2 \, dx + (\bar f - 
a)^2
.
\end{equation*}

Actually we will prove a slightly stronger inequality than~\eqref{eq:_goal}, as 
stated in the following lemma.

\begin{lemma}
\label{lem:ineq-main-lemma}
Let $U$ be a uniformly integrable set of smooth probability measures on 
$\overline\Omega$; then there exist $\varkappa > 0$ and $\sigma > 0$ such that
\begin{equation}
\label{eq:goal}
\int_{[u \ge \sigma]} u(x)\left((f(x, u(x)) - a)^2
+ |\nabla f(x, u(x))|^2 \right) dx
\ge \varkappa a^2
\end{equation}
for all $u \in U$ and $a \in \mathbb R$.
\end{lemma}

The proof is carried out in the subsequent lemmas.

Given a set~$M$ of integrable functions on~$\Omega$, let
\begin{equation*}
\omega_M(\delta)
= \sup \left \{
\int_A |u| \, dx \colon
u \in M, A \subset \Omega, |A| \le \delta
\right\}
\end{equation*}
be the modulus of integrability of~$M$.  Clearly, $\omega_M \colon [0, \infty) 
\to [0, \infty]$ is a nondecreasing function. Denote by
\begin{equation*}
\omega_M^- (t)
= \inf \{\delta \ge 0 \colon \omega_M(\delta) \ge t \}
\end{equation*}
its generalized inverse, cf. \cite{E13}.
Obviously,
\begin{equation*}
\text{$M$ is uniformly integrable} \Leftrightarrow \lim_{\delta \to 
+0}\omega_M(\delta) = 0
\Leftrightarrow \forall t > 0\colon \omega_M^- (t) > 0
.
\end{equation*}
\begin{remark}
\label{rem:ui}
Suppose that $f \to - \infty$ as $u \to \infty$ uniformly in~$x$.  Then if the 
entropy is bounded on~$U$, the set~$U$ is uniformly integrable.  This can be 
shown using a simple de la Vallée-Poussin argument.  First of all, note that 
by L'Hôpital's rule we have
\begin{equation*}
\lim_{u \to \infty} \frac{E(x, u)}{u}
=
\lim_{u \to \infty} (- f(x, u))
= \infty
,
\end{equation*}
where the limits are uniform in~$x$.  Given $\varepsilon > 0$ take $k > 0$ 
such that $u \le \varepsilon E(x, u)$ whenever $u \ge k$ and assume that $|A| 
\le \varepsilon$; then for any $u \in U$ we have
\begin{equation*}
\int_A u(x) \, dx
\le k|A| + \varepsilon \int_\Omega E(x, u(x)) \,dx
\le \left(k + \sup_{u \in U} \entropy(u)\right) \varepsilon
\end{equation*}
proving the uniform integrability.
\end{remark}

Given $c$, the equation
\begin{equation*}
f(x, \xi) = c
\end{equation*}
defines a positive function $m_c \in C(\overline \Omega)$, at least if $c$ is 
sufficiently close to~$0$. Clearly, $[u\geq m_c]$=$[f\le c]$, and similarly for other comparisons. 
\begin{remark}
\label{rem:ex-mc}
If $m_c$ exists for some $c > 0$, then $m_{c'}$ exists whenever $0 < c' \le 
c$; similarly, if $m_c$ exists for some $c < 0$, then $m_{c'}$ exists whenever 
$c < c' < 0$.
\end{remark}
\begin{remark}
It follows easily from the Mean Value Theorem that if $m_c$ exists for some $c 
> 0$, then
\begin{equation}
\label{eq:mc1}
\inf_\Omega (m - m_{c})
\ge
\frac{c}{\sup\limits_{m_c(x) \le \xi \le m(x)} | f_u(x, \xi) |}
,
\end{equation}
and if $m_c$ exists for some $c < 0$, then
\begin{equation}
\label{eq:mc2}
\inf_\Omega (m_{c} - m)
\ge
-
\frac{c}{\sup\limits_{m(x) \le \xi \le m_c(x)} | f_u(x, \xi) |}
.
\end{equation}
In the suprema above and in what follows we write $m_c(x) \le \xi \le m(x)$ 
for $\{(x, \xi) \colon m_c(x) \le \xi \le m(x)\}$, etc.  Clearly, the 
suprema in~\eqref{eq:mc1} and~\eqref{eq:mc2} are finite.
\end{remark}
\begin{remark}
\label{rem:ufu}
Note that
\begin{equation}
\label{eq:ufu}
\inf_{u > m} (uf)_u < 0
.
\end{equation}
Indeed, one only needs to observe that $(uf)_u = f + uf_u$ is uniformly 
negative both as $u \to m$ (since $m$ is uniformly positive and $f_u\Big|_{u = 
m}$ is uniformly negative) and as $u \to \infty$ (since so is $f$).
\end{remark}
\begin{lemma}
\label{lem:ineq62}
Suppose that $m_c$ exists for some $c > 0$; then for any $u \in U$ we have
\begin{equation}
\label{eq:ineq6a}
\int_{[m_{c} < u < m]} (m - u) \, dx
\le
\frac{1}{\inf\limits_{m_{c}(x) \le \xi \le m(x)}|f_u(x, \xi)|}
\int_{[m_c < u < m]} f \, dx
;
\end{equation}
likewise, if~$m_c$ exists for some $c < 0$, then
\begin{equation}
\label{eq:ineq6b}
\int_{[m < u < m_{c}]} (u - m) \, dx
\le
\frac{1}{\inf\limits_{m(x) \le \xi \le m_{c}(x)}|f_u(x, \xi)|}
\int_{[m < u < m_c]} f \, dx
.
\end{equation}
\end{lemma}
\begin{proof}
Both inequalities are easy consequences of the Mean Value Theorem if we take 
into account that $f(x, \xi) = 0$ when $\xi = m(x)$.
\end{proof}

\begin{lemma}
\label{lem:ineq2}
Suppose that $m_c$ is defined for some $c > 0$; then for any $u \in U$ we have
\begin{equation}
\label{eq:ineq2}
\big| [u > m] \big|
\ge
\omega_U^- \left(\inf_\Omega (m - m_c) \big| [ u \le m_c] \big| \right)
.
\end{equation}
\end{lemma}
\begin{proof}
We have:
\begin{align*}
1 & =
\int_{[u \le m_c]} u \, dx
+ \int_{[m_c < u \le m]} u \, dx
+ \int_{[u > m]} u \, dx
\\
& \le
\int_{[u \le m_c]} m_c \, dx
+ \int_{[m_c < u \le m]} m \, dx
+ \int_{[u > m]} u \, dx
\\
& = \int_{[u > m]} ( u - m) \, dx
- \int_{[u \le m_c]}( m - m_c) \, dx
+ \int_\Omega m \, dx
.
\end{align*}
The last integral equals~1, so
\begin{equation*}
\int_{[u > m]} ( u - m) \, dx
\ge
\int_{[u \le m_c]}( m - m_c) \, dx
\ge
\inf_\Omega (m - m_c) \big|[u \le m_c] \big|
.
\end{equation*}
Now using the positivity of $m$ we deduce
\begin{equation*}
\omega_U\left(\big|[u > m] \big| \right)
\ge
\int_{[u > m]} u \, dx
\ge
\int_{[u > m]} ( u - m) \, dx
\ge
\inf_\Omega (m - m_c) \big|[u \le m_c] \big|
,
\end{equation*}
and~\eqref{eq:ineq2} follows, observing that $\omega_U^-({\omega_U}(s)) \le s$. 
\end{proof}

\begin{lemma}
\label{lem:ineq2.5}
Suppose that $m_c$ is defined for some $c < 0$; then for any $u \in U$ we have
\begin{equation}
\label{eq:ineq2.5}
\big| [u < m] \big|
\ge
\frac{\inf_\Omega (m_c - m)}{\sup_\Omega m} \big| [ u \ge m_c] \big|
.
\end{equation}
\end{lemma}
\begin{proof}
Mimicking the proof of Lemma~\ref{lem:ineq2}, we obtain
\begin{equation*}
\int_{[u < m]} ( m - u) \, dx
\ge
\int_{[u \ge m_c]}( m_c - m) \, dx
\ge
\inf_\Omega (m_c - m) \big|[u \ge m_c] \big|
.
\end{equation*}
On the other hand, as $u$ is nonnegative, we have
\begin{equation*}
\int_{[u < m]} ( m - u) \, dx \le \sup_\Omega m \big| [u < m] \big|
,
\end{equation*}
and~\eqref{eq:ineq2.5} follows.
\end{proof}

\begin{lemma}
\label{lem:ineq1}
Let $c_0 < c_1$ and suppose that $m_{c_1}$ is defined; then for any $u \in U$ 
we have
\begin{equation}
\label{eq:ineq1}
\int_{[c_0 < f < c_1]} u | \nabla f|^2 \, dx
\ge
\frac{C_\Omega^2}{|\Omega|}
(c_1 - c_0)^2
\inf_\Omega m_{c_1}
\min
\left(
\big |[ f \le c_0 ] \big|,
\big |[ f \ge c_1 ] \big|
\right)^{2(d-1)/d}
.
\end{equation}
\end{lemma}
\begin{proof}
By monotonicity of $f$ we have $u \ge m_{c_1}$ on $[c_0 < f < c_1]$, so
\begin{align}
\int_{[c_0 < f < c_1]} u | \nabla f|^2 \, dx
& \ge
\inf_\Omega m_{c_1}
\int_{[c_0 < f < c_1]} | \nabla f|^2 \,dx
\notag
\\
& \ge
|\Omega|^{-1}
\inf_\Omega m_{c_1}
\left(\int_{[c_0 < f < c_1]} | \nabla f| \,d x \right)^2
.
\label{eq:tmp1}
\end{align}
In what follows, we use some basic results and concepts from the geometric 
measure theory, which can be found in \cite{Mag12}.  In particular,
the relative perimeter of a Lebesgue measurable set $A$ of locally finite 
perimeter with respect to $\Omega$
is defined as
\begin{equation*}
P(A; \Omega) = |\mu_A| (\Omega)
,
\end{equation*}
where $\mu_A:=\nabla 1_A$ is the Gauss-Green measure associated with~$A$. 
The support of~$\mu_A$ is contained in the topological boundary 
of~$A$.

Using the coarea formula, we have
\begin{align*}
\int_{[c_0 < f < c_1]} | \nabla f| \,d x
& =
\int_{-\infty}^{\infty} P\big( [f < t]; [c_0 < f < c_1] \big) \,d t
\\
& \ge
\int_{c_0}^{c_1} P\big( [f < t]; [c_0 < f < c_1] \big) \,d t
.
\end{align*}
The support of the Gauss--Green measure $\mu_{[f < t]}$ is contained in the 
topological boundary of the set $[f < t]$, so if $c_0 < t < c_1$, we see that 
the intersection of the support with $\Omega$ lies in $[c_0 < f < c_1]$.  
Consequently, we can take relative perimeter with respect to~$\Omega$ and 
proceed using the relative isoperimetric inequality (see, e.g., \cite{Mazja}) as follows:
\begin{align*}
\int_{[c_0 < f < c_1]} | \nabla f| \,d x
& \ge
\int_{c_0}^{c_1} P\big( [f < t]; \Omega \big) \,d t
\\
& \ge
C_\Omega \int_{c_0}^{c_1} \min \big( \big|[f < t]\big|, \big|[f \ge t]\big| 
\big)
^{(d-1)/d} \,d t
.
\end{align*}
The integrand can be estimated using the obvious inclusions
\begin{equation*}
[f < t] \supset [f \le c_0], \quad [f \ge t] \supset [f \ge c_1] \quad (c_0 < 
t < c_1)
,
\end{equation*}
and thus
\begin{equation*}
\int_{[c_0 < f < c_1]} | \nabla f| \,d x
\ge
C_\Omega (c_1 - c_0) \min \big( \big|[f \le c_0 ]\big|, \big|[f \ge c_1]\big| 
\big)
^{(d-1)/d}
.
\end{equation*}
Combining this with~\eqref{eq:tmp1}, we obtain~\eqref{eq:ineq1}. \end{proof}

\begin{lemma}
\label{lem:ineq3}
Let $c_0 < 0$ and $c_1 > 0$ and suppose that $m_{c_i}$ ($i = 0, 1$) are 
defined; then for any $u \in U$ we have
\begin{multline}
\label{eq:lemA}
\int_{[0 < f < c_1]} f \,dx \ge
\inf_{m_{c_1}(x) \le \xi \le m(x)}|f_u(x, \xi)|
\left(
-
\frac{c_0\big| [ u \ge m_{c_0} ] \big|}{\sup\limits_{m(x) \le \xi \le 
m_{c_0}(x)}| f_u(x, \xi) |}
-
\sup_\Omega m \big| [ u \le m_{c_1} ] \big|
\right)
\end{multline}
\end{lemma}
\begin{proof}
Since $u$ and $m$ are probability measures, we have
\begin{equation}
\label{eq:tmp4}
\int_{[u > m]} (u - m) \, dx = \int_{[u < m]} (m - u) \,dx
,
\end{equation} Let us estimate the sides of \eqref{eq:tmp4}.

For the left-hand side, we have
\begin{align*}
\int_{[u > m]} (u - m) \, dx
& \ge
\int_{[u \ge m_{c_0}]} (u - m) \, dx
\notag
\\
& \ge
\inf_\Omega (m_{c_0} - m) \big| [ u \ge m_{c_0} ] \big|
\notag
\\
& \ge
-
\frac{c_0\big| [ u \ge m_{c_0} ] \big|}{\sup\{ | f_u(x, \xi) | \colon m(x) \le 
\xi \le m_{c_0}(x)\}}
,
\end{align*}
where we have used~\eqref{eq:mc2}; for the right-hand side we have
\begin{align*}
\int_{[u < m]} (m - u) \, dx
& =
\int_{[u \le m_{c_1}]} (m - u) \, dx
+
\int_{[m_{c_1} < u < m]} (m - u) \, dx
\\
& \le
\sup_\Omega m \big| [ u \le m_{c_1} ] \big|
+
\frac{1}{\inf\limits_{m_{c_1}(x) \le \xi \le m(x)}|f_u(x, \xi)| \}}
\int_{[m_{c_1} < u < m]} f \, dx
,
\end{align*}
where we have used~\eqref{eq:ineq6a}.  Comparing the estimates, we arrive 
at~\eqref{eq:lemA}.
\end{proof}

Now we are in the position to prove Lemma~\ref{lem:ineq-main-lemma} for small 
negative $a$.

\begin{lemma}
\label{lem:ineq4}
Suppose that $m_c$ exists for $|c| \le \delta$; then there exist $a_\delta \in 
(-\delta, 0)$ and $\varkappa_\delta > 0$ such that~\eqref{eq:goal} holds for 
all $a \in (a_\delta, 0)$ and $u \in U$ with $\varkappa = \varkappa_\delta$ 
and any positive $\sigma \le \inf_\Omega m_\delta$.
\end{lemma}
\begin{proof}
Fix $u \in U$, $\sigma \le \inf_\Omega m_\delta$, and $a \in (a_\delta, 0)$, 
the constant $a_\delta$ to be defined below.  We examine the possible 
alternatives and in each of them, we find a suitable value for 
$\varkappa_\delta$.

Observe that in~$\Omega$,
\begin{equation*}
f < \delta \Leftrightarrow u > m_\delta \Rightarrow u > \sigma
.
\end{equation*}

Consider the following partition of~$\Omega$:
\begin{equation}
\Omega = [f \ge \delta] \cup [ a/2 < f < \delta ] \cup [f \le a/2]
.
\end{equation}
Clearly, at least one set on the right-hand side has volume $\ge |\Omega|/3$.

If $\big|[f \ge \delta] \big| \ge |\Omega|/3$, it follows from 
Lemma~\ref{lem:ineq2} that $\big|[f \le 0]\big| \ge \sigma_\delta$ with 
$\sigma_\delta > 0$ independent of~$u$ and~$a$.  Then Lemma~\ref{lem:ineq1} 
guarantees the estimate
\begin{equation*}
\int_{[u > \sigma]} u | \nabla f|^2 \, dx
\ge
\int_{[0 < f < \delta]} u | \nabla f|^2 \, dx
\ge C_\delta
\ge
\frac{C_\delta}{\delta^2} a^2
\end{equation*}
with $C_\delta > 0$ independent of~$u$ and~$a$, so~\eqref{eq:goal} holds with 
$\varkappa = \varkappa_\delta' := C_\delta/\delta^2$.

If $\big| [a/2 < f < \delta ] \big| \ge |\Omega|/3$, we have the following 
simple lower bound on the first term on the left-hand side 
of~\eqref{eq:goal}:
\begin{align*}
\int_{[u > \sigma]} u (f - a)^2 \, dx
& \ge
\int_{[m_\delta < u < m_{a/2}]} u (f - a)^2 \, dx
\\
& \ge
\frac{\inf_\Omega m_\delta}{|\Omega|}
\left(
\int_{[a/2 < f < \delta]} (f - a) \, dx
\right)^2
\\
& \ge
\frac{\inf_\Omega m_\delta}{4|\Omega|}
\big| [a/2 < f < \delta ] \big|^2 a^2
\\
& \ge
\frac{|\Omega|\inf_\Omega m_\delta}{36}
a^2
=: \varkappa_\delta'' a^2
,
\end{align*}
so~\eqref{eq:goal} holds with $\varkappa = \varkappa_\delta''$.

It remains to assume that $\big|[ f \le a/2] \big| \ge |\Omega|/3$ and $s := 
\big| [f \ge \delta ] \big| < |\Omega|/3$.  Using Lemma~\ref{lem:ineq1} with 
$c_1 = \delta$ and $c_0 = a/2$, we obtain
\begin{equation*}
\int_{[a/2 < f < \delta]} u |\nabla f|^2 \, dx
\ge C_\delta s^{2(d-1)/d}
.
\end{equation*}
Of course, the right-hand side is a lower bound for the left-hand side 
of~\eqref{eq:goal}, so if $s \ge |a|^{d/(d-1)}$, the inequality holds with 
$\varkappa = \varkappa_\delta''' = C_\delta$.

Thus, assume that
\begin{equation*}
s <|a|^{d/(d-1)}
.
\end{equation*}
Now we evoke Lemma~\ref{lem:ineq3} with $c_0 = a/2$ and $c_1 = \delta$.    
Taking the supremum and infimum of~$|f_u|$ on the right-hand side 
of~\eqref{eq:lemA} over the larger set $\Omega \times [-\delta\leq f \le \delta]$, we 
ensure that these extreme values are independent of~$a$ and the inequality still holds, i.~e.\ 
we have
\begin{equation*}
\int_{[0 < f < \delta]} f \, dx
\ge A_\delta a - B_\delta s
\ge \left (A_\delta - B_\delta |a|^{1/(d-1)} \right) |a|
\ge \frac{A_\delta}{2} |a|
\end{equation*}
given that $|a| < - a_\delta := \min( (A_\delta/(2B_\delta))^{d - 1}, 
\delta)$.  Then the first term on the left-hand side of~\eqref{eq:goal} admits 
the estimate
\begin{align*}
\int_{[u > \sigma]} u (f - a)^2 \, dx
& \ge
\inf_\Omega m_\delta
\int_{[m_\delta < u < m]} (f - a)^2 \, dx
\\
& \ge
\frac{\inf_\Omega m_\delta}{|\Omega|}
\left(
\int_{[0 < f < \delta]} f \, dx
\right)^2
\\
& \ge
\frac{A_\delta^2\inf_\Omega m_\delta}{4|\Omega|}
a^2
=: \varkappa_\delta'''' a^2
.
\end{align*}
To complete the proof, it suffices to take $\varkappa_\delta = 
\min(
\varkappa_\delta'
,
\varkappa_\delta''
,
\varkappa_\delta'''
,
\varkappa_\delta''''
)$.
\end{proof}

\begin{lemma}
\label{lem:ineq5}
Let $a \ge 0$ and $c > 0$, and suppose that $m_c$ exists; then for any $u \in 
U$ we have
\begin{equation}
\label{eq:ineq5}
\int_{[u > m]} u(f - a)^2\,dx \ge
\left( 
\frac{
\inf_{u > m} (uf)_u
}{
\sup_{m_{c} \le u \le m} |f_u|
}
\right)^2
c^2
\big|[f > c]\big|^2
.
\end{equation}
\end{lemma}
\begin{proof}
Let us again estimate both sides of \eqref{eq:tmp4}.

On one hand, we have
\begin{align*}
\int_{[u < m]} (m - u) \, dx
& \ge \int_{[u < m_c]} (m - u) \, dx
\\
& \ge \inf_\Omega (m - m_c) \big|[u < m_c]\big|
\\
& \ge
\frac{c\big|[u < m_c]\big|}{\sup\limits_{m_c(x) \le \xi \le m(x)} | f_u(x, \xi) 
|}
,
\end{align*}
where we take advantage of~\eqref{eq:mc1}.

Before estimating the right-hand side of~\eqref{eq:tmp4}, observe that if $\xi 
> m$, we can use the Mean Value Theorem and get
\begin{equation*}
\xi |f(x, \xi)|
= |\xi f(x, \xi) - m(x) f(x, m(x))|
\ge\left| \inf_{u > m} (uf)_u \right|(\xi - m(x))
,
\end{equation*}
where the modulus of the infimum is uniformly positive by Remark~\ref{rem:ufu}.  Now, setting 
$\xi = u(x)$, we have
\begin{align*}
\int_{[u > m]} (u - m) \, dx
& \le
\left| \left( \inf_{u > m} (uf)_u \right) \right|^{-1}
\int_{[u > m]} u|f| \, dx
\\
& \le
\left| \left( \inf_{u > m} (uf)_u \right) \right|^{-1}
\int_{[u > m]} u|f - a| \, dx
\\
& =
\left| \left( \inf_{u > m} (uf)_u \right) \right|^{-1}
\int_{\Omega} u|f - a| 1_{[u > m]}(x) \, dx
\\
& \le
\left| \left( \inf_{u > m} (uf)_u \right) \right|^{-1}
\left(\int_{\Omega} u(f - a)^21_{[u > m]}(x) \, dx\right)^{\frac 12}
\\
& =
\left| \left( \inf_{u > m} (uf)_u \right) \right|^{-1}
\left(\int_{[u > m]} u(f - a)^2 \, dx\right)^{\frac 12}
,
\end{align*} since $u$ is a probability measure.
Comparing this with the above estimate of the left-hand side 
of~\eqref{eq:tmp4}, we recover~\eqref{eq:ineq5}.
\end{proof}

Now we prove Lemma~\ref{lem:ineq-main-lemma} for small positive~$a$.

\begin{lemma}
\label{lem:ineq61}
Suppose that $\delta > 0$ is such that $m_{\delta/2}$ is defined; then there 
exists $\varkappa_\delta > 0$ such that inequality \eqref{eq:goal} holds with 
$\varkappa = \varkappa_\delta$ and any positive $\sigma \le \inf_\Omega 
m_{\delta/2}$ for all $u \in U$ and $a \in (0, \delta)$.
\end{lemma}
\begin{proof}
Fix $\sigma \le \inf_\Omega m_{\delta/2}$, $u \in U$, and $a \in (0, \delta)$.  
Observe that in~$\Omega$,
\begin{equation*}
f < \frac\delta2 \Leftrightarrow u > m_{\delta/2} \Rightarrow u > \sigma
.
\end{equation*}

By Remark~\ref{rem:ex-mc}, $m_{a/2}$ is defined.  Consider the partition
\begin{equation*}
\Omega = \left[ f > \frac a2 \right] \cup \left[ f \le \frac a2 \right]
.
\end{equation*}
Obviously, at least one of the sets on the right-hand side has volume $\ge 
|\Omega|/2$.

Suppose that
\begin{equation*}
\left| \left[ f > \frac a2 \right] \right| \ge \frac{|\Omega|}{2}
.
\end{equation*}
Taking into account inequality~\eqref{eq:ineq5} for $c = a/2$ and observing 
that
\begin{equation*}
\sup_{m_{a/2} \le u \le m} |f_u|
\ge
\sup_{m_{\delta/2} \le u \le m} |f_u|
\end{equation*}
with the right-hand side independent of~$a$, we obtain
\begin{equation*}
\int_{[u > \sigma]} u(f - a)^2 \, dx
\ge
\int_{[u > m]} u(f - a)^2 \, dx
\ge
\varkappa_\delta' a^2
\end{equation*}
with some constant~$\varkappa_\delta'$ independent of~$a$ and~$u$.

If, on the other hand, we have
\begin{equation*}
\left| \left[ f \le \frac a2 \right] \right| \ge \frac{|\Omega|}{2}
,
\end{equation*}
then
\begin{align*}
\int_{[u > \sigma]} u (f - a)^2 \, dx
& \ge
\int_{[f \le a/2]} u (f - a)^2 \, dx
\\
& \ge
\left(\frac 14 \left| \left[ f \le \frac a2 \right] \right|\inf_\Omega m_{a/2} 
\right)
a^2
\\
& \ge
\left(\frac 18 |\Omega| \inf_\Omega m_{\delta/2} \right)
a^2
=:
\varkappa_\delta''
a^2
\end{align*}
with $\varkappa_\delta''$ independent of~$u$ and~$a$.

To complete the proof, it suffices to take $\varkappa_\delta = 
\min(\varkappa_\delta', \varkappa_\delta'')$.
\end{proof}

\begin{lemma}
\label{lem:ineq6}
Suppose that $\delta > 0$ is such that $m_{\delta}$ is defined; then there 
exists $\varkappa_\delta > 0$ such that inequality \eqref{eq:goal} holds with 
$\varkappa = \varkappa_\delta$ and any positive $\sigma \le \inf_\Omega 
m_\delta$ for all $u \in U$ and $a < -2 \delta$.
\end{lemma}
\begin{proof}
Given $a < -2 \delta$ and $u \in U$, write
\begin{equation*}
|\Omega|
= \left|\left [ f \le \frac a2 \right ]\right|
+ \left| \left [\frac a2 < f \le 0 \right]\right|
+ \big|[0 < f < \delta ] \big|
+ \big| [ f \ge \delta] \big|
=: s_1 + s_2 + s_3 + s_4
.
\end{equation*}

Clearly,
\begin{equation}
\label{eq:tmp5}
\max s_i \ge \frac{|\Omega|}{4}
.
\end{equation}

It follows from Lemmas~\ref{lem:ineq2} and~\ref{lem:ineq2.5} that a lower 
bound on $\big|[f \ge \delta]\big| = s_4$ yields a lower bound on $\big|[f < 
0]\big| \le s_1 + s_2$ and a lower bound on $ s_1 = \big| [ f \le a/2 ]\big| 
\le \big| [ f \le - \delta ]\big|$ yields a lower bound on $ \big| [ f > 0 ] 
\big| = s_3 + s_4$.  Together with~\eqref{eq:tmp5} this implies that at least 
one of the following inequalities hold:
\begin{gather*}
s_2 \ge \frac{|\Omega|}{4}, \quad
s_3 \ge \frac{|\Omega|}{4},
\\
\min(s_1 + s_2, s_4) \ge 2c_\delta, \quad
\min(s_3 + s_4, s_1) \ge 2c_\delta
,
\end{gather*}
where 
$c_\delta > 0$ is independent of~$u$ and~$a$.  Assuming for definiteness that 
$c_\delta < |\Omega|/4$, we easily check that either
\begin{equation}
\label{eq:tmp6}
\min
\left(
\left |\left[ f \le \frac a2 \right] \right|,
\big |[ f \ge \delta ] \big|
\right)
=
\min(s_1, s_4) \ge c_\delta
\end{equation}
or
\begin{equation}
\label{eq:tmp7}
\left| \left [\frac a2 < f < \delta \right]\right|
=
s_2 + s_3 \ge c_\delta
.
\end{equation}

On the set $[a/2 < f < \delta]$ we clearly have $u > \sigma$. Thus, 
if~\eqref{eq:tmp6} is true, using Lemma~\ref{lem:ineq1} we obtain
\begin{equation*}
\int_{[u >\sigma]} u |\nabla f|^2 \, dx
\ge
\int_{[a/2 < f < \delta]} u |\nabla f|^2 \, dx
\ge
4
\varkappa_\delta'
\left(\delta - \frac a2 \right)^2
\ge
\varkappa_\delta' a^2
.
\end{equation*}
If, on the other hand, \eqref{eq:tmp7} holds, note that $a/2 < f < \delta$ 
implies $f - a > -a/2 > 0$, and estimate
\begin{equation*}
\int_{[u > \sigma]} u(f - a)^2 \, dx
\ge
\int_{[a/2 < f < \delta]} u(f - a)^2 \, dx
\ge
\frac{a^2}{4} \inf_\Omega m_\delta
\left| \left [\frac a2 < f < \delta \right]\right|
\ge \varkappa_\delta'' a^2
.
\end{equation*}
Thus, one can take $\varkappa_\delta = \min(\varkappa_\delta', 
\varkappa_\delta'')$.
\end{proof}

\begin{lemma}
\label{lem:ineq7}
Suppose that $\delta > 0$ is such that $m_{\delta}$ is defined; then there 
exists $\varkappa_\delta > 0$ such that inequality \eqref{eq:goal} holds with 
$\varkappa = \varkappa_\delta$ and any $\sigma \le \inf m_\delta$ for all $u 
\in U$ and $a \ge 2 \delta$.
\end{lemma}
\begin{proof}
The partition
\begin{equation*}
\Omega = [ f < \delta ] \cup [f \ge \delta]
\end{equation*}
ensures that either $\big|[ f < \delta ]\big| \ge |\Omega|/2$ or $\big|[ f \ge 
\delta ]\big| \ge |\Omega|/2$.  In the latter case Lemma~\ref{lem:ineq2} 
guarantees a lower bound on $\big| [f \le 0] \big|$ and hence on $\big| [f < 
\delta] \big|$.  Either way, we can write
\begin{equation*}
\big| [f < \delta] \big| \ge s_\delta
,
\end{equation*}
where $s_\delta$ is independent of $a$ and $u$.

As $f < \delta$ implies $u > \sigma$ and $f - a < \delta - a \le - a/2$, we 
have the estimate
\begin{align*}
\int_{[u > \sigma]} u (f - a)^2 \, dx
& \ge
\int_{\Omega} u (f - a)^2 1_{[f < \delta]}(x) \, dx
\\
& \ge
\left(
\int_\Omega u |f - a|1_{[f < \delta]}(x) \, dx
\right)^2
\\
& =
\left(
\int_{[u > m_\delta]} u |f - a| \, dx
\right)^2
\\
& \ge
\left(
\frac 14
s_\delta
\inf_\Omega m_\delta
\right)^2
a^2
\end{align*}
and \eqref{eq:goal} follows.
\end{proof}
Now we can assemble the proof of Lemma~\ref{lem:ineq-main-lemma} from 
established particular cases.
\begin{proof}[Proof of Lemma~\ref{lem:ineq-main-lemma}]
Take $\delta_1 > 0$ such that $m_c$ exists whenever $|c| \le \delta_1$.  By 
Lemma~\ref{lem:ineq4}, there exist $\varkappa_1 > 0$, $\sigma_1 > 0$, and $a_1 
\in (-\delta_1 ,0)$ such that \eqref{eq:goal} holds with $\varkappa = 
\varkappa_1$ and $\sigma = \sigma_1$ for all $u \in U$ and $a \in (a_1, 
0)$.  Set $\delta_2 = - a_1$.  This is a suitable value of~$\delta$ for 
Lemma~\ref{lem:ineq61}, so we conclude that~\eqref{eq:goal} holds with 
$\varkappa = \varkappa_2$ and $\sigma = \sigma_2$ for $u \in U$ and $a \in (- 
\delta_2, \delta_2)$ and, moreover, $m_c$ is defined whenever $|c| \le 
\delta_2$.  Now in order to find~$\varkappa$ and~$\sigma$ such 
that~\eqref{eq:goal} holds for all $u \in U$ and all real~$a$, it suffices to 
evoke Lemmas~\ref{lem:ineq6} and~\ref{lem:ineq7} with $\delta = \delta_2/3$.
\end{proof}

\section{Applications}

\subsection{Spherical gradient flows} \label{s:gf}

Let~$\Omega$ be an open connected bounded domain 
in~$\mathbb R^d$ with sufficiently smooth boundary and let~$\nu$ be the outward 
unit normal along~$\partial \Omega$.  We are interested in nonnegative 
solutions of
\begin{align}
\partial_t u & = -\Div (u \nabla f) + u\left( f - \int_\Omega uf\, dx \right), 
& (x, t) & \in \Omega \times (0, \infty),
\label{eq:p1}
\\
u \frac{\partial f}{\partial \nu} & = 0, & (x, t) & \in \partial \Omega \times 
(0, \infty),
\label{eq:p2}
\\
u & = u^0, & (x, t) & \in \Omega \times {0},
\label{eq:p3}
\\
u & \ge 0, \ \int_\Omega u \,dx = 1,
& (x, t) & \in \Omega \times (0, \infty)
.
\label{eq:p4}
\end{align}
Here $u$ is the unknown function and $f = f(x, u(x,t))$ is a known nonlinear scalar
function of $x$ and $u$.  The initial data~$u^0$ is a probability density.

For the sake of brevity we will denote
\begin{equation*}
\bar f = \int_\Omega u f \, dx
.
\end{equation*}

\begin{remark} \label{period}
The Neumann boundary condition \eqref{eq:p2} can be substituted with the space-periodic one without affecting the validity of the results of this section. 
\end{remark}

Throughout Section \ref{s:gf}, we make the following assumptions about the nonlinearity~$f$. Some of the results do not require all of these assumptions: it will be explicitly indicated where relevant.
\begin{gather}
f \in
C^2( \overline\Omega \times (0, \infty))
\cap L^1_\text{loc}(\overline \Omega \times [0, \infty))
,
\label{eq:f1a}
\\
uf, uf_x
\in C(\overline \Omega \times [0, +\infty))
,
\label{eq:f1b}
\\
f_u < 0
,
\label{eq:f2}
\\
|f(x, u)|  \le g_1(u)
\quad
\text{a.~a. } u > 0; \ g_1 \in L^1_\text{loc}[0, \infty)
,
\label{eq:f5}
\\
u |f_u(x, u)| + u |f_{xu}(x, u)| \le g_2(u)
\quad
\text{a.~a. } u > 0; \ g_2 \in L^1_\text{loc}[0, \infty)
,
\label{eq:f55}
\\
(uf_x)\big|_{u = 0} = 0
,
\label{eq:f1d}
\\
\text{either $f_x = 0$ for large $u$} \text{\quad or\quad} \lim_{u \to \infty} 
f(x, u) = - \infty \ \forall x \in \overline \Omega
,
\label{eq:large-u}
\\
\text{either $f_x = 0$ for small $u$} \text{\quad or\quad} \lim_{u \to +0} 
f(x, u) = \infty \ \forall x \in \overline \Omega
,
\label{eq:small-u}
\\
u\left[f_x^2+ (uf_{xu})^2+(uf_u)^2\right]= O(1) \quad \text{as $u \to 0$ uniformly in $x \in \Omega$}
,
\label{eq:O1}
\\
u f_{uu} = O(f_u) \quad \text{as $u \to 0$ uniformly in $x \in 
\Omega$}.
\label{eq:O2}
\end{gather}

Assumption  \eqref{eq:f2} ensures non-strict parabolicity of the problem.  
The remaining assumptions are technical.  It is easy to check (see 
\cite[Remark 3.4]{KV17}) that~\eqref{eq:large-u} and~\eqref{eq:small-u} ensure that given 
$v \in L^\infty_+(\Omega)$ bounded away from~$0$, there exist $m_{c_1}$ and 
$m_{c_2}$ (this notation was introduced in the beginning of Section \ref{s:ineq}) such that $m_{c_1} \le v \le m_{c_2}$ a.~e.\ in~$\Omega$. In particular, taking $v\equiv\frac 2 {|\Omega|}$ and $v\equiv\frac 1 {2|\Omega|}$ in this observation, we infer existence of $m_{c_1}$, $m_{c_2}$  such that $$\int_{\Omega} m_{c_1}\,dx\leq \frac 1 2,\ \int_{\Omega} m_{c_2}\,dx\geq 2.$$ This implies (cf. Remark \ref{rem:ex-mc}) existence and uniqueness of a $C^2$-smooth probability density $m \colon 
\overline \Omega \to (0, \infty)$ such that $f(x, m(x))$ is constant on~$\overline\Omega$. Since problem \problem{} does not change after adding constants to $f$, without loss of generality we will assume that \begin{equation} f(x, m(x))=0.\label{impm}\end{equation}

Let us introduce the energy and entropy functionals for equation~\eqref{eq:p1} 
as well as the notion of weak solution.

Bound~\eqref{eq:f55} ensures that
\begin{equation*}
\Phi(x, u) = - \int_0^u \xi f_u(x, \xi) \,d \xi,
\quad
\Psi(x, u) = \int_0^u \Phi(x, \xi) \,d \xi
\end{equation*}
are well defined and belong to
$C^1(\overline \Omega \times [0, \infty))$, whereas
\begin{align*}
\Phi(x, 0) & = \Psi(x, 0) = 0,
& \Phi_u & = - u f_u,
\\
\Phi_x & = - \int_0^u \xi f_{xu}(x, \xi) \,d \xi,
& \Psi_u & = \Phi,
\\
\Phi_{uu} & = - (uf_u)_u,
& \Phi_{xu} & = - uf_{xu}
.
\end{align*}
Note that both~$\Phi$ and~$\Psi$ are nonnegative and strictly increase with 
respect to~$u$.

By~\eqref{eq:f55}, the superposition operator $L^\infty_+ \to L^\infty$ 
associated with~$\Phi$ is bounded, i.~e. if $u$ is a nonnegative function 
of~$x$ and, possibly, $t$, then an $L^\infty$-bound on~$u$ is translated into 
an $L^\infty$-bound on $\Phi(\cdot, u(\cdot))$.  The same is true of $\Phi_x$ 
and $\Psi$.

In accordance with~\cite{KV17}, we call the functional
\begin{equation*}
\energy(u) = \int_\Omega \Psi(x, u(x)) \, dx
\end{equation*}
the \emph{energy} of problem \problem{}.

Define
\begin{equation}
E(x, u) = - \int_{m(x)}^u f(x, \xi) \,d\xi \label{eq:entr0}
.
\end{equation}
It follows from~\eqref{eq:f5} that~$E$ is well-defined and continuous on 
$\overline \Omega \times [0, \infty)$.  Moreover, $E \ge 0$ and $E(x, u) = 0$ 
if and only if $u = m(x)$, and the superposition operator associated with~$E$ 
is bounded in $L^\infty_+ \to L^\infty_+$.  Thus, for $u \in 
L_+^\infty(\Omega)$ we can define the relative \emph{entropy} of 
equation~\eqref{eq:p1} as follows:
\begin{equation} \label{eq:entr}
\entropy (u) = \int_\Omega E(x, u(x)) \,dx
.
\end{equation}

\begin{lemma}
\label{lem:prop-class}
Let $u$ be a classical solution of~\problem{} on~$[0, T]$. Then $u$ satisfies
\begin{enumerate}
\renewcommand\labelenumi{(\roman{enumi})}
\item
the \emph{energy identity}
\begin{equation}
\label{eq:energy-id}
\partial_t
\energy(u)
= - \int_{\Omega} | \nabla \Phi |^2 \,d x
+ \int_{\Omega} (\Phi_x + uf_x) \cdot \nabla \Phi \,d x
+ \int_{\Omega} u(f - \bar f)\Phi \,d x
\quad t > 0
;
\end{equation}
\item
the \emph{entropy dissipation identity}
\begin{equation}
\label{eq:entropy-diss-class}
\partial_t \entropy(u) = - \int_\Omega u((f - \bar f)^2 + | \nabla f |^2) \cd 
x
\quad t > 0
;
\end{equation}
\item
the bounds
\begin{equation}
\label{eq:linfbound}
\inf_\Omega f(x, u^0(x)) \le f(x, u(x, t)) \le \sup_\Omega f(x, u^0(x))
\quad (x, t) \in \Omega \times (0, \infty)
.
\end{equation}
\end{enumerate}
\end{lemma}
\begin{proof}
Straightforward computation proves (i) and (ii).

Let us prove the first inequality in~\eqref{eq:linfbound}.  Assume that the 
infimum is finite, because otherwise there is nothing to prove; denote it 
by~$c$.  It follows from~\eqref{eq:large-u} that the function $m_c \colon 
\overline \Omega \to \mathbb R$ satisfying $f(x, m_c(x)) \equiv c$ is defined.  
We have
\begin{equation*}
\partial_t \int_\Omega (u - m_c)_+ \, dx
= \int_\Omega \theta(u - m_c) \partial_t u \, dx
,
\end{equation*}
where
\begin{equation*}
\theta(s) =
\begin{cases}
1 & \text{if } s > 0, \\
0 & \text{if } s \le 0
\end{cases}
\end{equation*}
is the Heaviside step function.  Substituting the right-hand side of the 
equation for $\partial_t u$, we obtain
\begin{align*}
\partial_t \int_\Omega (u - m_c)_+ \, dx
& =
- \int_\Omega \theta(u - m_c) \operatorname{div}(u\nabla f) \, dx
+ \int_\Omega \theta(u - m_c) u(f - \bar f) \, dx
\\
& =: -I_1 + I_2
.
\end{align*}

Writing
\begin{equation*}
I_1 = \int_\Omega \theta(u - m_c) \operatorname{div} (u \nabla f - m_c \nabla 
f(x, m_c(x))) \, dx
,
\end{equation*}
we can use~\cite[Lemma 3.1]{KV17} and conclude that~$I_1 \ge 0$ (though the 
lemma is proved for~$C^\infty$ functions, it holds for $C^2$ functions by 
density).

Now, if
\begin{equation*}
\int_{[u \ge m_c]} u \,dx = 0
,
\end{equation*}
we have $u \le m_c$ a.~e.\ in $\Omega$ and consequently, $I_2 = 0$.  
Otherwise,
\begin{equation*}
I_2
= \int_{[u \ge m_c]} u \, dx
\left(
\frac{
\int_{[u \ge m_c]} uf \, dx
}{
\int_{[u \ge m_c]} u \, dx
}
- \bar f
\right)
\ge 0
,
\end{equation*}
since the average of~$f$ with weight~$u$ over the set $[u \ge m_c] = [f \le 
c]$ is no greater than the weighted average over the whole~$\Omega$.

Thus, we see that
\begin{equation*}
\partial_t \int_\Omega (u - m_c)_+ \, dx \le 0
,
\end{equation*}
and as this integral equals~$0$ at $t = 0$, it equals~$0$ for any~$t$, which 
is equivalent to $u \le m_c$ and to the first inequality 
in~\eqref{eq:linfbound}.

The second inequality in~\eqref{eq:linfbound} is proved in the same way.
\end{proof}

The integral on the right-hand side of~\eqref{eq:entropy-diss-class} is called 
the \emph{entropy production}.  We denote it by~$\entropyproduction(u)$, so 
that~\eqref{eq:entropy-diss-class} can be written as
\begin{equation} \label{eq:eep-df}
\partial_t \entropy(u) = - \entropyproduction(u)
.
\end{equation}

\begin{remark}
We can extend the definition of the entropy production to functions $u \in L_+^\infty(\Omega)$ such that 
$\Phi(\cdot, u(\cdot)) \in H^1(\Omega)$ by the formula
\begin{equation*}
\entropyproduction(u) = \int_\Omega u(f - \bar f)^2 \cd x
+ \int_{[ u > 0]} \frac 1 u | - \nabla \Phi + \Phi_x + uf_x |^2 \cd x
,
\end{equation*}
where the second integral on the right-hand side may be infinite. This is relevant for the weak solutions which will be introduced in Definition \ref{d:ws}.
\end{remark}

Let $Q_T := \Omega \times (0, T)$.

\begin{lemma}
\label{lem:deriv}
If $u$ is a classical solution of~\eqref{eq:p1}--\eqref{eq:p4} on~$[0, T]$ 
satisfying
\begin{equation*}
\| u \|_{L^\infty(Q_T)} \le R
,
\end{equation*}
then
\begin{equation*}
\label{eq:deriv}
\| \partial_t \Phi(u) \|_{[C([0, T]; W^{1, \infty}(\Omega))]^*} \le C(R, T)
\end{equation*}
with $C(R, T) > 0$ independent of $u$. 
\end{lemma}
\begin{proof}
For a given test function $\psi \in C([0, T]; W^{1, \infty}(\Omega))$ we have
\begin{align*}
|\langle \partial_t \Phi(u), \psi \rangle|
& = \left|\int_{Q_T} \psi \Phi_u \partial_t u \, dx \, dt \right|
\\
& = \left|\int_{Q_T} \psi \Phi_u (- \operatorname{div}(u \nabla f) + u(f - 
\bar f)) \, dx \, dt \right|
\\
& \le \| \psi \|_{C([0, T]; W^{1, \infty}(\Omega))}
\Bigg(
\int_{Q_T} u | \nabla \Phi_{u} ||\nabla f| \,dx \, dt
+\int_{Q_T} u |\Phi_{u}| |\nabla f|\, dx \, dt
\\
&
\qquad
+\int_{Q_T} u | \Phi_u ||f - \bar f| \,dx \, dt
\Bigg)
=\| \psi \|_{C([0, T]; W^{1, \infty}(\Omega))} (I_1 + I_2 + I_3)
.
\end{align*}
Our goal is to show that the integrals $I_k$ are bounded from above.

By \eqref{eq:O1}, \eqref{eq:O2} there exist $C \ge 0$ and $\varepsilon > 
0$ both independent of~$u$ such that
\begin{gather}
u|f_x|^2 \le C
,
\label{eq:tmp10}
\\
u^3 |f_{xu}|^2 \le C
,
\label{eq:tmp11}
\\
u|\Phi_u|^2 \le C
\label{eq:tmp12}
,
\\
|\Phi_{uu}| \le C |f_u|
\label{eq:tmp13}
\end{gather}
whenever $0 < u < \varepsilon$.  Moreover, if we allow~$C$ to depend on~$T$, 
we can assume that~\eqref{eq:tmp10}--\eqref{eq:tmp12} hold on $\overline 
\Omega \times (0, T]$, since the left-hand sides are continuous and $|f_u|$ is 
positive.

For $I_1$ we have
\begin{align*}
I_1
& =
\int_{Q_T} u \big( | \Phi_{uu}| | \nabla u| + |\Phi_{xu}|\big ) | \nabla f| \, 
dx \,dt
\\
& \le
\int_{Q_T} u \big( C| f_u | | \nabla u| +  u|f_{xu}|\big ) | \nabla f| \, dx 
\,dt
\\
& \le
\int_{Q_T} u \big( C| f_u \nabla u + f_x| + C|f_x| + u|f_{xu}| \big ) | \nabla 
f| \, dx \,dt
\\
& \le
C\int_{Q_T} u |\nabla f|^2 \, dx \,dt
+
\left( 2\int_{Q_T} (Cu|f_x|^2 + u^3|f_{xu}|^2) \, dx \,dt \right)^{1/2}
\left(\int_{Q_T} u |\nabla f|^2 \, dx \,dt\right)^{1/2}
\\
& \le C' \int_0^T \left(\entropyproduction(u) + \sqrt{\entropyproduction(u)} 
\right) \, dt
\\
& \le C' \int_0^T \entropyproduction(u) \, dt
+ C'\sqrt T \left( \int_0^T \entropyproduction(u) \, dt \right)^{1/2}
.
\end{align*}
As we assume an upper bound on $u$, the integral
\begin{equation*}
\int_0^T \entropyproduction(u) \, dt = \entropy(0) - \entropy(T)
\end{equation*}
is bounded, so we see that~$I_1$ is bounded uniformly in~$u$.

Further, we have
\begin{equation*}
I_2+I_3
\le
\left(
\int_{Q_T} u |\Phi_{u}|^2 \, dx \, dt
\right)^{1/2}
\left(2
\int_{Q_T} u( |\nabla f|^2+|f-\bar f|^2)  dx \, dt
\right)^{1/2}
\le
C''
\left(
\int_0^T \entropyproduction(u) \, dt
\right)^{1/2}
,
\end{equation*}
where the last term is bounded.
\end{proof}

\begin{lemma}
\label{lem:ex-smooth}
For 
any smooth probability density $u^0 \colon \overline \Omega \to (0, \infty)$ satisfying the 
non-flux boundary condition, problem \problem{} has a 
classical solution.
\end{lemma}
\begin{proof} Equation~\eqref{eq:p1} can be cast in the form
\begin{equation} \label{eq:par}
\partial_t u = - uf_u \Delta u - \nabla u \cdot (f_x + f_u \nabla u)
- u (f_{xx} + 2f_{xu} \cdot \nabla u + f_{uu} | \nabla u |^2 - f+\bar f)
.
\end{equation}
Since the initial data $u^0$ is strictly positive, any classical solution $u$ is \emph{a priori} bounded away 
from~$0$ and $\infty$.  Indeed, evoking \cite[Remark 3.4]{KV17}, we can find $m_{c_1}$ and $m_{c_2}$ strictly positive
such that $c_2 \le 0 \le c_1$ and
\begin{equation*}
m_{c_1}(x) \le u^0(x) \le m_{c_2}(x) \quad (x \in \Omega)
.
\end{equation*}
Then \eqref{eq:linfbound} and \eqref{eq:f2}   yield
\begin{equation*}
m_{c_1}(x) \le u(x, t) \le m_{c_2}(x), \quad (x, t) \in \Omega \times (0, 
\infty).
\end{equation*}
Hence we can avoid 
degeneracies or singularities in \eqref{eq:par} and apply \cite[Theorem 13.1]{Am93} to secure existence and uniqueness of a maximal weak
solution $\tilde u$ in the sense of Amann. This solution is global in time 
provided we can control its norm in a certain Sobolev space. Viewing $$\bar 
f(t):=\int_{\Omega} \tilde u(x,t)f(x,\tilde u(x,t))\, dx$$ as a given 
coefficient, we ``deactivate'' the nonlocal term  in  \eqref{eq:par}. 
Bootstrapping and employing the results of \cite[Sections 14 and 15]{Am93}, we 
can improve the regularity of $\bar f$ (as a function of time) and that of 
$\tilde u$ (as a function of time and space).  
Integrating~\eqref{eq:p1} with $u = \bar u$ in space, we see that the mass is 
conserved along the flow.  We conclude that $\tilde u$ is actually a global 
smooth solution to \problem{}.  \end{proof}

\begin{definition} \label{d:ws}
Let $u^0 \in L^\infty(\Omega)$ be a probability density.  A function~$u \in 
L_+^\infty(Q_T)$ is called a \emph{weak solution} 
of~\eqref{eq:p1}--\eqref{eq:p4} on $[0, T]$ if $\int_\Omega u(x,t)\,d 
x=1$ for a.a. $t\in(0,T)$, $\Phi(\cdot, u(\cdot)) \in L^2(0, T; 
H^1(\Omega))$, and
\begin{equation}
\label{eq:def1}
\int_0^T
\int_\Omega
(u \partial_t \varphi + ( - \nabla \Phi + \Phi_x + u f_x) \cdot \nabla \varphi 
+ (f - \bar f) u \varphi) \,d x \,d t
=
\int_\Omega u^0(x) \varphi(x, 0) \,d x
\end{equation}
for any function $\varphi \in C^1(\overline \Omega \times [0, T])$ such that 
$\varphi(x, T) = 0$.
A function $u \in L^\infty_\text{loc}([0, \infty); L_+^\infty(\Omega))$ is 
called a \emph{weak solution} of \eqref{eq:p1}--\eqref{eq:p4} on $[0, 
\infty)$ if for any~$T > 0$ it is a weak solution on~$[0, T]$.
\end{definition}

\begin{theorem}[Existence of weak solutions]
\label{th:ex}
Suppose that~$f$ satisfies~\eqref{eq:f1a}--\eqref{impm}.  Then for any 
probability density ~$u^0 \in L^\infty_+(\Omega)$ there exists a weak 
solution $u \in L^\infty_+(\Omega \times (0, \infty))$ of problem 
\eqref{eq:p1}--\eqref{eq:p4} enjoying the following properties:
\begin{enumerate}
\item $u$ satisfies the energy inequality
\begin{equation}
\label{eq:energy}
\partial_t
\energy(u)
\le
\int_{\Omega}
\big (
- | \nabla \Phi |^2
+ (\Phi_x + uf_x) \cdot \nabla \Phi
+ u(f-\bar f)\Phi
\big )
\cd x
\end{equation}
in the sense of measures and
\begin{equation}
\label{eq:init-energy}
\esslimsup_{t \to +0} \energy(u(t)) \le \energy(u^0);
\end{equation}
\item
$u$ satisfies the entropy dissipation 
inequality
\begin{equation}
\label{eq:entropy}
\partial_t \entropy(u)
\le - \entropyproduction(u)
\end{equation}
in the sense of measures and
\begin{equation}
\label{eq:init-entropy}
\esssup_{t > 0} \entropy(u(t)) \le \entropy(u^0)
.
\end{equation}
\end{enumerate}
\end{theorem}
\begin{proof}
It is easy to see that we can approximate the initial data~$u^0$ by smooth and strictly positive probability densities
$u^0_n$ satisfying the boundary condition in such a way that
\begin{gather}
\| u^0_n \|_{L^\infty(\Omega)} \le C
,
\\
u^0_n \to u^0 \quad \text{weakly$^*$ in }L^\infty(\Omega)\ \text{and a.e. in}\ \Omega
\label{eq:sp2}
,
\\
\energy(u^0_n) \to \energy(u^0)
,
\\
\entropy(u^0_n) \to \entropy(u^0)
.
\end{gather} The last two convergences can be secured using the Lebesgue 
Dominated Convergence Theorem. Let $u_n$ be the classical solution 
starting from $u^0_n$, which exists by Lemma~\ref{lem:ex-smooth}.

Put
\begin{align*}
f_n & = f(x, u_n(x, t)), & f_{xn} & = f_x(x, u_n(x, t)), \\
\Phi_n & = \Phi(x, u_n(x, t)), & \Phi_{xn} & = \Phi_x(x, u_n(x, t)), \\
\Psi_n & = \Psi(x, u_n(x, t)), & E_n & = E(x, u_n(x, t)).
\end{align*}
Given $T > 0$, by Lemma~\ref{lem:prop-class} the sequence $\{u_n\}$ is bounded 
in $L^\infty(Q_T)$, and so are the sequences $\{u_n f_n\}$, $\{u_n f_{xn}\}$, 
$\{\Phi_n\}$, $\{\Phi_{xn}\}$, $\{\Psi_n\}$, and $\{E_n\}$.  It follows from 
the energy identity~\eqref{eq:energy-id} that
\begin{equation} \label{eq:star}
\partial_t \energy(u_n) \le - \frac 12 \int_\Omega |\nabla \Phi_n|^2 \,dx + C
,
\end{equation}
whence the integral
\begin{equation*}
\int_{Q_T}|\nabla \Phi_n|^2 \, dx
\le 2\big(\energy(u_n^0) - \energy(u_n(T)) + CT\big)
\end{equation*}
is bounded, i.~e.\ $\{\Phi_n\}$ is bounded in $L^2(0, T; H^1(\Omega))$.  By 
Lemma~\ref{lem:deriv} the derivatives $\{\partial_t \Phi_n\}$ are bounded in 
$[C(0, T; W^{1, \infty}(\Omega))]^*$.  Hence, \cite[Corollary 
7.9]{Rou} implies that $\{\Phi_n\}$ is compact in $L^2(Q_T)$.  This is 
true for any~$T$, so $\{\Phi_n\}$ is compact in $L^2_\text{loc}([0, 
\infty); L^2(\Omega))$ and there is no loss of generality that $\Phi_n 
\to \phi$ in this space and a.~e.\ in $\Omega \times (0, \infty)$.

Fix $(x, t) \in \Omega \times (0, \infty)$ such that
\begin{equation*}
\Phi(x, u_n(x, t)) = \Phi_n(x, t) \to \phi(x, t)
.
\end{equation*}
Assuming that $\| u_n \|_{L^\infty(\Omega \times (0, \infty))} \le R$ and 
taking into account that $\Phi$ increases in~$u$, we have $\Phi_n(x, t) \le 
\Phi(x, R)$ and so $0 \le \phi(x, t) \le \Phi(x, R)$.  As~$\Phi$ is continuous 
in~$u$, there exists a unique $u(x, t) \in [0, R]$ such that $\Phi(x, u(x, t)) 
= \phi(x, t)$, and as the inverse of~$\Phi$ with respect to~$u$ is continuous 
in~$u$ as well, we have $u_n(x, t) \to u(x, t)$.  Thus, we have defined a 
function $u \in L^\infty_+(\Omega \times (0, \infty))$ such that for any $T > 0$ 
we have
\begin{gather}
\label{eq:sp-conv}
\left.
\begin{aligned}
u_n & \to u \\
u_n f_n & \to uf \\
u_n f_{xn} & \to uf_x \\
\Phi_n & \to \Phi \\
\Phi_{xn} & \to \Phi_x \\
\Psi_n & \to \Psi
\end{aligned}
\right\}
\begin{tabular}{l}
a.~e.\ in  $Q_T$, \\
strongly in any $L^p(Q_T)$, $1 \le p < \infty$, \\
weakly* in $L^\infty(Q_T)$, \\
and in the sense of distributions,
\end{tabular}
\\
\bar f_n \to \bar f
\\
\nabla \Phi_n \to \nabla \Phi \quad \text{weakly in } L^2(Q_T)
.
\label{eq:sp-conv-a}
\end{gather}
where we write $\Phi$ for $\Phi(\cdot, u(\cdot))$, etc.

The function $u$ is a weak solution of \eqref{eq:p1}--\eqref{eq:p4} on $[0, 
T]$ as it follows from~\eqref{eq:sp2} and 
\eqref{eq:sp-conv}--\eqref{eq:sp-conv-a} that one can pass to the limit in the 
weak setting for the approximate solution
\begin{multline}
\label{eq:sp3}
\int_0^T
\int_\Omega
(u_n \partial_t \varphi + ( - \nabla \Phi_n + \Phi_{xn} + u_n f_{xn}) \cdot 
\nabla \varphi + (f_n - \bar f_n) u_n \varphi) \,d x \,d t
\\
=
\int_\Omega u^0_n(x) \varphi(x, 0) \,d x
,
\end{multline}
where $\varphi$ is an admissible test function.

In order to show that $u$ satisfies the energy inequality on $[0, T]$ in the 
sense of measures, we take a smooth nonnegative test function $\chi \in 
C^\infty(\R)$ vanishing outside of $[0, T]$ and rewrite the energy identity 
from Lemma~\ref{lem:prop-class} in the sense of measures for the approximate 
solutions:
\begin{multline*}
- \int_{Q_T} \Psi_n \chi'(t) \,d x \,d t
=
- \int_{Q_T} |\nabla \Phi_n|^2 \chi(t) \,d x \,d t
\\
+ \int_{Q_T} \chi(t) (\Phi_{xn} + u_n f_{xn}) \cdot \nabla \Phi_n \,d x 
\,d t
+ \int_{Q_T} u_n (f_n - \bar f_n) \Phi_n \chi(t) \,d x cd t
\end{multline*}
Here one can use convergences~\eqref{eq:sp-conv} to pass to the limit in all 
the terms but for the first one on the right-hand side.  Further,  
\eqref{eq:sp-conv-a} implies that $\sqrt \chi \, \nabla \Phi_n \to \sqrt 
\chi \, \nabla \Phi$ weakly in $L^2(Q_T)$, so
\begin{equation*}
\int_{Q_T} \chi | \nabla \Phi |^2 \,d x \,d t
\le
\liminf_{n \to \infty}
\int_{Q_T} \chi | \nabla \Phi_n |^2 \,d x \,d t
,
\end{equation*}
and the energy inequality follows.

Let us check~\eqref{eq:init-energy}.  By \eqref{eq:star}, the approximate solutions 
satisfy
\begin{equation*}
\esssup_{t \in (0, \varepsilon)} \energy(u_n(t)) \le \energy(u_n^0) + 
C\varepsilon
.
\end{equation*}
It follows from~\eqref{eq:sp-conv} that
\begin{equation*}
\energy(u_n) \to \energy(u) \quad \text{weakly* in } L^\infty(0, \varepsilon)
,
\end{equation*}
so we get
\begin{align*}
\esssup_{t \in (0, \varepsilon)} \energy(u(t))
& \le
\liminf_{n \to \infty} \esssup_{t \in (0, \varepsilon)} \energy(u_n(t))
\\
& \le
\lim_{n \to \infty} \energy(u_n^0) + C\varepsilon
\\
& =\energy(u^0) + C\varepsilon
.
\end{align*}
Now sending $\varepsilon \to 0$ we recover~\eqref{eq:init-energy}.

Let us show that $u$ satisfies the entropy dissipation inequality on $[0, T]$ 
in the sense of measures.  Let $\chi \in C^\infty$ be a smooth nonnegative 
test function vanishing outside of $[0, T]$.  By Lemma~\ref{lem:prop-class}, 
the approximate solutions satisfy the entropy dissipation identity.  It can be 
recast in the sense of measures as follows:
\begin{multline*}
- \int_{Q_T} E_n \chi'(t) \,d x \,d t
=
- \int_{Q_T} \chi(t) u_n (f_n - \bar f_n)^2 \,d x \,d t
\\
- \int_{u_n > 0} \frac{\chi(t)}{u_n} | - \nabla \Phi_n + \Phi_{xn} + u_n 
f_{xn}|^2 \,d x \,d t
.
\end{multline*}
Consequently, for any $\delta > 0$ we have
\begin{multline}
\label{eq:sp4}
- \int_{Q_T} E_n \chi'(t) \,d x \,d t
\le
- \int_{Q_T} \frac{\chi(t)}{\max(u_n, \delta)} (u_n (f_n - \bar f_n))^2 
\,d x \,d t
\\
- \int_{Q_T} \frac{\chi(t)}{\max(u_n, \delta)}
| - \nabla \Phi_n + \Phi_{xn} + u_n f_{xn}|^2 \,d x \,d t
.
\end{multline}
Observe that
\begin{gather}
\label{eq:mery}
\frac{\chi(t)}{\max(u_n, \delta)}
\to
\frac{\chi(t)}{\max(u, \delta)}
\begin{tabular}{l}
a.~e.\ in  $Q_T$, \\
strongly in any $L^p$, $1 \le p < \infty$, \\
and weakly* in $L^\infty(Q_T)$,
\end{tabular}
\\ \label{eq:vk1}
v_n:=- \nabla \Phi_n + \Phi_{xn} + u_n f_{xn}
\to - \nabla \Phi + \Phi_{x} + u f_{x} \quad \text{weakly in 
$L^2(\Omega)$}.
\end{gather}
In \cite[claim (3.24)]{KV17} it was proved that
\begin{multline}
\int_{Q_T} \frac{\chi(t)}{\max(u, \delta)}
| - \nabla \Phi + \Phi_{x} + u f_{x}|^2 \,d x \,d t
\\ \leq \liminf_{n\to \infty}
\int_{Q_T} \frac{\chi(t)}{\max(u_n, \delta)}
| - \nabla \Phi_n + \Phi_{xn} + u_n f_{xn}|^2 \,d x \,d t
\label{f:leon}\end{multline}
and using~\eqref{eq:sp-conv}, we pass to the limit in~\eqref{eq:sp4} obtaining
\begin{multline*}
- \int_{Q_T} E \chi'(t) \,d x \,d t
\le
- \int_{Q_T} \frac{\chi(t)}{\max(u, \delta)} (u (f-\bar f))^2 \,d x \,d t
\\
- \int_{Q_T} \frac{\chi(t)}{\max(u, \delta)}
| - \nabla \Phi + \Phi_{x} + u f_{x}|^2 \,d x \,d t
.
\end{multline*}
On the set $\{(x, t) \in Q_T \colon u(x, t) = 0\}$ we have $uf_x = 0$ (by 
virtue of~\eqref{eq:f1d}), $\Phi_x = 0$ and $\Phi = 0$, whence also $\nabla 
\Phi = 0$ a.~e.\ on this set.  Thus, we can write
\begin{multline*}
- \int_{Q_T} E \chi'(t) \,d x \,d t
\le
- \int_{Q_T} \frac{\chi(t)}{\max(u, \delta)} (u (f-\bar f))^2 \,d x \,d t
\\
- \int_{u > 0} \frac{\chi(t)}{\max(u, \delta)}
| - \nabla \Phi + \Phi_{x} + u f_{x}|^2 \,d x \,d t
\end{multline*}
Letting $\delta \to 0$, by Beppo Levi's Theorem we obtain the entropy 
inequality.

Inequality~\eqref{eq:init-entropy} is proved in the same way 
as~\eqref{eq:init-energy} given that it holds for the approximate solutions.
\end{proof}

\begin{theorem}[Entropy-entropy production inequality]
\label{th:eep}
Suppose that~$f$ satisfies~\eqref{eq:f1a}--\eqref{eq:f5}, \eqref{impm}. Assume 
that the second of the alternatives in \eqref{eq:large-u} holds, and 
the limit is uniform w.r.t. $x$.  Let $U \subset L^\infty(\Omega)$ be 
a set of probability densities such that for any $u \in U$, we have 
$\Phi(\cdot, u(\cdot)) \in H^1(\Omega)$ and
\begin{equation}
\label{eq:eep-a}
\sup_{u \in U} \entropy(u) < \infty
.
\end{equation}
Then there exists $C_U$ such that
\begin{equation}
\label{eq:eep-b}
\entropy(u) \le C_U \entropyproduction(u) \quad (u \in U)
.
\end{equation}
\end{theorem}
\begin{proof}
Let us show that \eqref{eq:goalfa} holds with~$U$ merely satisfying the 
hypotheses of Theorem \ref{th:eep}.  According to Remark~\ref{rem:ui}, condition~\eqref{eq:eep-a} ensures the 
uniform integrability of~$U$. As explained before Lemma \ref{lem:ineq-main-lemma}, it suffices to ensure that 
inequality~\eqref{eq:_goal} holds for~$U$.

Given $u \in U$, we use the construction presented in the proof 
of~\cite[Theorem~1.7]{KV17} and approximate the function $\Phi(\cdot, 
u(\cdot))$ with smooth functions $\Phi_n$ in such a way that
\begin{equation*}
\Phi_n \to \Phi(\cdot, u(\cdot)) \quad \text{in } H^1 \text{and a.~e.\ in } 
\Omega,
\end{equation*}
while the functions $u_n \in C^2(\Omega)$ satisfying $\Phi(x, u_n(x)) = \Phi_n(x)$ 
are well-defined and
\begin{equation}
\left.
\begin{array}{l}
\|u_n\|_{L^\infty} \le C
,
\\
u_n \to u \quad \text{a.~e.\ in } \Omega
.
\end{array}
\right\}
\end{equation}
There is no loss of generality in assuming that~$u_n$ are probability 
measures, since we can normalize them taking into account that 
\begin{equation*}
\|u_n\|_{L^1(\Omega)} \to \|u\|_{L^1(\Omega)} = 1
.
\end{equation*}

By Lemma~\ref{lem:ineq-main-lemma} with $a = \bar f_n$, we have
\begin{equation}
\label{eq:tmp8}
\int_{[u_n \ge \sigma]} u_n\left((f_n - \bar f_n)^2
+ |\nabla f_n|^2 \right) dx
\ge \varkappa \bar f_n^2
\end{equation}
with $\sigma > 0$ and $\varkappa >0$ independent of~$n$, where as usual $f_n$ 
stands for $f(x, u_n(x))$, etc.  Inequality~\eqref{eq:tmp8} can be written as
\begin{equation*}
\int_{\Omega}
\left(1_{[u_n \ge \sigma]}u_n(f_n - \bar f_n)^2
+
\frac{1_{[u_n \ge \sigma]} }{u_n}| - \nabla \Phi_n + \Phi_{xn} + u_nf_{xn}|^2 
\right) \, dx
\ge \varkappa \bar f_n^2
.
\end{equation*}
As the integrand vanishes whenever $u_n < \sigma$, one can pass to the limit 
as $n \to \infty$ (cf.~\cite{KV17}).  Observing that
\begin{equation*}
\limsup_{n \to \infty} 1_{[u_n \ge \sigma]}(x) \le 1_{[u \ge \sigma]}(x) \quad 
\text{a.~e.\ in } \Omega,
\end{equation*}
and employing the  Reverse Fatou Lemma for products \cite{KV19A} we obtain
\begin{equation*}
\int_{\Omega}
\left(1_{[u \ge \sigma]}u(f - \bar f)^2
+
\frac{1_{[u \ge \sigma]} }{u}| - \nabla \Phi + \Phi_{x} + uf_x|^2 \right) \, 
dx
\ge \varkappa \bar f^2
,
\end{equation*}
which is stronger than~\eqref{eq:_goal}.
\end{proof}

\begin{theorem}[Convergence to equilibrium]
\label{th:convergence}
Suppose that~$f$ satisfies~\eqref{eq:f1a}--\eqref{eq:f5}, \eqref{impm}. Assume 
that the second of the alternatives in \eqref{eq:large-u} holds, and 
the limit is uniform w.r.t. $x$.    Let~$u$ be a weak solution 
of~\eqref{eq:p1}--\eqref{eq:p4} with the initial data $u^0 \in 
L^\infty_+(\Omega)$, $\int_\Omega u^0=1$.  Suppose that~$u$ satisfies 
the entropy dissipation inequality~\eqref{eq:entropy} and
inequality~\eqref{eq:init-entropy}. Then~$u$ exponentially converges to~$m$ in 
the sense of entropy:
\begin{equation}
\label{eq:convergence-a}
\entropy(u(t)) \le \entropy(u^0) \mathrm{e}^{-\gamma t} \quad \text{a.~a.\ } t 
> 0
,
\end{equation}
where $\gamma > 0$ can be chosen uniformly over initial data satisfying
\begin{equation}
\label{eq:convergence-b}
\entropy(u^0) \le C
\end{equation}
with some $C > 0$.
\end{theorem}
\begin{proof}
As the entropy decreases along the solution, the set
\begin{equation*}
U = \left\{ u \in L^\infty_+(\Omega) \colon \int_{\Omega} u=1,\ \entropy(u) \le C \right\}
\end{equation*}
is invariant under the flow generated by the problem: more precisely, $u(t) 
\in U$ for a.~a.\ $t \ge 0$.  Let~$C_U$ be correspondent constant in the 
entropy-entropy production inequality granted by Theorem~\ref{th:eep}.  
Combining the entropy dissipation and entropy-entropy production inequalities 
for a given solution $u$, we obtain
\begin{equation*}
\partial_t \entropy(u(t)) \le - C_U^{-1} \entropy(u(t)) \quad \text{a.~a. } t 
> 0
.
\end{equation*}
Letting $e(t) = \entropy(u(t)) \mathrm{e}^{C_U^{-1} t}$, we see that 
$\partial_t e(t) \le 0$ in the sense of measures, whence $e$ a.~e.\ coincides 
with a nonincreasing function.  Moreover,
\begin{equation*}
\esssup_{t > 0} e(t)
= \esslimsup_{t \to 0} e(t)
= \esslimsup_{t \to 0} \entropy(u(t)) \mathrm{e}^{C_U^{-1}t}
\le \entropy(u^0)
\end{equation*}
yielding \eqref{eq:convergence-a} with $\gamma = C_U^{-1}$.
\end{proof}

\begin{remark}
\label{rem:linfty}
Theorem~\ref{th:eep} holds without assuming the second alternative 
in~\eqref{eq:large-u}.  However, in this case the set~$U$ should be 
uniformly integrable.   This is clear from the proof.
Theorem~\ref{th:convergence} is valid for the solutions constructed in 
Theorem~\ref{th:ex} assuming the first alternative 
in~\eqref{eq:large-u} instead of the second, but
the constant~$\gamma$ in~\eqref{th:convergence} would depend on 
$\|u^0\|_{L^\infty(\Omega)}$.
It suffices to observe that for large~$C$ the set
\begin{equation*}
U = \left\{ u \in L^\infty_+(\Omega) \colon \int_{\Omega} u=1,\ 
\|u\|_{L^\infty(\Omega)} < C \right\}
\end{equation*}
is invariant under the flow.  Indeed, assume that~$C$ is so large that $f(x, 
C) = c$ does not depend on~$x$.  Then for any data $u^0 \in U$ we 
clearly have $f(x, u^0(x)) > c$, and this inequality is preserved 
along the flow.  This follows from Lemma~\ref{lem:prop-class} for the
classical solutions and by approximation for the weak solutions.
\end{remark}

\subsection{Nonlinear Fokker-Planck equations and generalized log-Sobolev inequalities} \label{s:logs}

Let us return for a moment to the setting \eqref{eq:eep-E}-\eqref{eq:eep-f}. Note that we still do not assume any displacement convexity. Theorem \ref{th:ineq-main} immediately implies
\begin{corollary}[Generalized log-Sobolev]
\label{co:ineq-main}
Let $U$ be a uniformly integrable set of smooth probability measures on 
$\overline\Omega$, which satisfy the weighted Poincar\'e inequality \begin{equation}\int_\Omega u(x)\left(g(x)-\int_\Omega ug\right)^2\,dx 
\leq c \int_\Omega u(x)|\nabla g(x)|^2 \, dx\end{equation} with a uniform constant $c$ independent of $u\in U$ and $g\in C^1(\overline \Omega)$. Then 
\begin{equation}
\label{eq:was}
\int_\Omega E(x, u(x)) \, dx
\le C \int_\Omega u(x)|\nabla f(x, u(x))|^2 \, dx,
\end{equation}
where the constant $C$ may depend on $U$ but is independent of $u \in U$.
\end{corollary}

Consider the nonlinear Fokker-Planck equation
\begin{align}
\partial_t u & = -\Div (u \nabla f), 
& (x, t) & \in \Omega \times (0, \infty),
\label{eq:p1w}
\\
u \frac{\partial f}{\partial \nu} & = 0, & (x, t) & \in \partial \Omega \times 
(0, \infty),
\label{eq:p2w}
\\
u & = u^0, & (x, t) & \in \Omega \times {0},
\label{eq:p3w}
\\
u & \ge 0, \ \int_\Omega u \,dx = 1,
& (x, t) & \in \Omega \times (0, \infty)
.
\label{eq:p4w}
\end{align}
Here $u$ is the unknown function and $f = f(x, u(x))$ is a known nonlinear scalar
function of $x$ and $u$, satisfying \eqref{eq:f1a}, \eqref{eq:f2}.  The initial data~$u^0$ is a probability density. As in Remark \ref{period}, \eqref{eq:p2w} can be replaced by the periodic boundary conditions. 

For simplicity, assume that $u^0$ is bounded away from $0$ and $\infty$.  Then the behaviour of $f$ at $u=0,\infty$ is not important, and we do not lose any generality in assuming existence and uniqueness of a $C^2$-smooth probability density $m \colon 
\overline \Omega \to (0, \infty)$ such that $f(x, m(x))=0$ (cf. Section \ref{s:gf}). Define the relative entropy $\mathcal E$ by \eqref{eq:entr0}, \eqref{eq:entr}. The existence of a unique classical solution (which is smooth for $t>0$) for such initial data is straightforward. 

\begin{theorem}[Convergence to equilibrium without reaction]
\label{th:convergencew}
Assume \eqref{eq:f1a}, \eqref{eq:f2}, \eqref{impm}.  Let~$u$ be a solution 
of~\eqref{eq:p1w}--\eqref{eq:p4w} with the initial data $u^0 \in 
L^\infty_+(\Omega)$, $\int_\Omega u^0=1$, $
\kappa_1\le u^0 \le \kappa_2
$ a.e. in $\Omega$
with some $\kappa_1,\kappa_2 > 0$.  Then~$u$ exponentially converges to~$m$ in 
the sense of entropy:
\begin{equation}
\label{eq:convergence-aww}
\entropy(u(t)) \le \entropy(u^0) \mathrm{e}^{-\gamma t} 
,
\end{equation}
where $\gamma=\gamma(\kappa_1,\kappa_2) > 0$ is independent of $u^0$. 
\end{theorem}

\begin{remark} A particular case of Theorem \ref{th:convergencew} when $f(x,u)=\frac {\rho(x)}{u^{r+1}}$, $\rho(x)$ is a given function bounded away from $0$ and $\infty$, $r=cst>0$, with $\Omega$ being a torus or a bounded convex domain, has recently been established in \cite{iak16,IS18}. The corresponding Wasserstein gradient flow is related to the  
problem of  quantisation  for  probability  measures. In this situation it is even possible to prove the exponential convergence merely if certain Lebesgue norms of $u^0$ and $\frac 1 {u^0}$ are finite, since under this hypothesis any solution instantaneously \cite{IS18} becomes bounded away from $0$ and $\infty$. This assumption at least visually resembles the definition of the Muckenhoupt weights \cite{St93}, which are known  \cite{FKS82} to satisfy the Poincar\'e inequality. In view of Corollary \ref{co:ineq-main}, it is plausible that similar exponential convergence results hold for general entropies when $u^0$ is, for instance, merely a Muckenhoupt weight. \end{remark}

Let us sketch the proof of Theorem \ref{th:convergencew}. Since the behaviour of $f$ at $u=0,\infty$ is not relevant, we may assume \eqref{eq:large-u} and~\eqref{eq:small-u}. Using \cite[Remark 3.4]{KV17}, we find $m_{c_1}$ and $m_{c_2}$ strictly positive such that $c_2 \le 0 \le c_1$ and
\begin{equation*}
m_{c_1}(x) \le \kappa_1 \le \kappa_2 \le m_{c_2}(x) \quad (x \in \Omega)
.
\end{equation*}  Now observe that problem \eqref{eq:p1w}--
\eqref{eq:p3w} (without fixing the mass to be $1$) admits a comparison principle: $u^0_1(x)\leq u^0_2(x)$ a.e. in $\Omega$ implies $u_1(x,t)\leq u_2(x,t),\ t>0$. This follows from \cite[Lemma 3.1]{KV17} by mimicking the proof of  \cite[Lemma 3.2]{KV17}. Hence, the set $U$ of smooth probability measures satisfying $m_{c_1}\le u \le m_{c_2}$
is invariant under the flow generated by this problem. Corollary 
\ref{co:ineq-main} guarantees that \eqref{eq:was} holds for this $U$. A 
standard Wasserstein entropy-entropy production argument \cite{Vil03} yields 
\eqref{eq:convergence-aww}.

\subsection{Unbalanced transportation inequalities} \label{sec:tala} For simplicity, here we restrict ourselves to the spatially periodic setting, although everything seems to work for bounded convex domains. Let $\Mm$ and $\Pp$ be the sets of Radon and probability measures, resp.,  on the flat torus $\mathbb T^d$.  The Hellinger-Kanto\-rovich distance, cf. \cite{KMV16A,LMS16,LMS18,CP18,CPSV18,Rez15}, on $\Mm$ and the spherical Hellinger-Kantoro\-vich distance, cf. \cite{LM17,BV18}, on $\Pp$ can be introduced as follows.
\begin{definition}[Conic distance] \label{d:metr}
Given two Radon measures $\rho_0,\rho_1\in \Mm$ we define
\begin{equation}\label{e:mini}
d_{HK}^2(\rho_0,\rho_1)=\inf\limits_{\mathcal{A}(\rho_0,\rho_1)} \int_0^1\left(\int_{\mathbb T^d}(|v_t|^2+|\alpha_t|^2)\cd\rho_t\right)\cd t,
\end{equation}
where the admissible set $\mathcal{A}(\rho_0,\rho_1)$ consists of all $(\rho_t,\alpha_t,v_t)_{t\in [0,1]}$ such that
\begin{equation*}
\left\{ 
\begin{array}{l} 
\rho\in\mathcal{C}_w([0,1];\Mm),\\
\rho|_{t=0}=\rho_0;\quad \rho|_{t=1}=\rho_1,\\
(\alpha,v)\in L^2(0,T;L^2(\cd\rho_t)\times L^2(\cd\rho_t)^d),\\
\partial_t\rho_t+\dive(\rho_t v_t)=\rho_t\alpha_t \quad \mbox{in the weak sense}.
\end{array}
\right.
\end{equation*}
\end{definition}

\begin{definition}[Spherical distance] \label{d:metr1}
Given probability measures $\rho_0,\rho_1\in \Pp$ we define
\begin{equation}\label{e:mini1}
d_{HKS}^2(\rho_0,\rho_1)=\inf\limits_{\mathcal{A}_1(\rho_0,\rho_1)} \int_0^1\left(\int_{\mathbb T^d}(|v_t|^2+|\alpha_t|^2)\cd\rho_t\right)\cd t,
\end{equation}
where the admissible set $\mathcal{A}_1(\rho_0,\rho_1)$ consists of all $(\rho_t,\alpha_t,v_t)_{t\in [0,1]}$ such that
\begin{equation*}
\left\{ 
\begin{array}{l} 
\rho\in\mathcal{C}_w([0,1];\Pp),\\
\rho|_{t=0}=\rho_0;\quad \rho|_{t=1}=\rho_1,\\
(\alpha,v)\in L^2(0,T;L^2(\cd\rho_t)\times L^2(\cd\rho_t)^d),\\
\partial_t\rho_t+\dive(\rho_t v_t)=\rho_t\alpha_t \quad \mbox{in the weak sense}.
\end{array}
\right.
\end{equation*}
\end{definition}
The relation between the two distances is given by the fact that $(\Mm,d_{HK})$ is a metric cone over $(\Pp,d_{HKS})$ \cite{LM17,BV18} (see, e.g., \cite{Bur,BH99} for the abstract definition of a metric cone). The definitions above and the classical Benamou-Brenier formula immediately imply that \begin{equation} \label{compa} d_{HK}( \rho_0,\rho_1)\leq d_{HKS}( \rho_0,\rho_1) \leq W_2( \rho_0,\rho_1)\end{equation} for all $ \rho_0,\rho_1\in \Pp(\mathbb T ^d)$, where $W_2$ stands for the quadratic Wasserstein distance.

The conventional transportation inequality \eqref{eq:goalfa4}  (also known as 
Talagrand's inequality \cite{OV00,CMcV03,CGH04}) estimates the Wasserstein 
distance by strictly displacement convex relative entropies. Here we present 
similar inequalities for the spherical distance $d_{HKS}$ and the conic 
distance $d_{HK}$, but for a much wider class of entropies. In view of 
\eqref{compa}, Theorem~\ref{th:ineq-tala} is interesting merely for the 
entropies which are not strictly geodesically convex in the Wasserstein space. 

\begin{remark} \label{rem:en} In Section \ref{s:gf} we defined the relative entropy $\entropy(u)$ for 
bounded probability distributions, but we can actually use any absolutely continuous probability measure $u$, although the entropy may become infinite. Moreover, the relative entropy can be defined in the same way for distributions of any mass, and without assuming that the implicit function $m$ defined by \eqref{impm} is a probability measure (cf. \cite{KV17}). \end{remark}

\begin{theorem}[Spherical Talagrand inequality]
\label{th:ineq-tala}
Suppose that~$f$ satisfies~\eqref{eq:f1a}--\eqref{eq:f5}, \eqref{impm}. Assume 
that the second of the alternatives in \eqref{eq:large-u} holds, and 
the limit is uniform w.r.t. $x\in \mathbb T^d$. Let $u^0\in 
L^1(\mathbb T^d)$ be a probability density with 
$\entropy(u^0)<\infty$. Then \begin{equation}
\label{eq:tala}
d_{HKS}^2(u^0,m) \leq C\, \entropy(u^0),
\end{equation}
with $C$ independent of $u^0$.
\end{theorem}

\begin{proof} The proof is an adaptation of the Otto-Villani strategy \cite{OV00}. We first observe that it suffices to prove the theorem when $u^0$ is smooth and strictly positive. Indeed, every $u^0 \in L^1(\mathbb T^d)$ with finite entropy can be approximated with bounded (from above and below) functions $\chi_k \circ u^0$, where $\chi_k(s)=\max (k^{-1}, \min (s,k))$. Since both $d_{HK}$ and $W_2$ metrize the weak topology of $\Pp(\mathbb T ^d)$, \eqref{compa} implies that $d_{HKS}$ metrizes the same topology. This fact and Beppo Levi's Theorem imply that both sides of \eqref{eq:tala} are continuous w.r.t. our approximation. Each of the $\chi_k \circ u^0$ can be approximated by smooth bounded (from above and below) functions, cf. the proof of Theorem \ref{th:ex}, so that both sides of \eqref{eq:tala} are continuous w.r.t. the approximation. The claim follows by a diagonal argument with renormalization of the masses in order to have an approximating sequence of probability distributions. 

Since the left-hand side is always bounded by $\pi^2$ \cite{BV18}, we only need to consider the case when $\entropy(u^0)$ is bounded, say, by $1$. Consider the classical solution $u$ to problem \eqref{eq:p1}, \eqref{eq:p3}, \eqref{eq:p4} on $\mathbb T^d$ (cf. Lemma \ref{lem:ex-smooth} and Remark \ref{period}), and let $f=f(x,u(x,t))$. As in the proof of Theorem \ref{th:convergence}, with the help of Theorem \ref{th:eep} we can find a constant $C_1$ such that \begin{equation}
\label{eq:eep-1}
\entropy(u_t) \le C_1 \entropyproduction(u_t) ,\ t\geq 0
.
\end{equation}

A simple scaling observation shows that the triple $$(u_{s+th},h (f_{s+th}-\overline f_{s+th}),h \nabla f_{s+th})$$ belongs to the admissible set $\mathcal{A}_1(u_s,u_{s+h})$, $s\geq 0,$ $h> 0$. By the definition of the distance, $$d_{HKS}(u_s,u_{s+h})\leq h\sqrt{\int_0^1\left(\int_{\mathbb T^d}(|f_{s+th}-\overline f_{s+th}|^2+|\nabla f_{s+th}|^2) u_{s+th}\cd x\right)\cd t}.$$ As $h\to 0$, the square root on the right-hand side converges to $\entropyproduction(u_s)$, and we infer\begin{equation} \frac d {dh}\Big|^+_{h=0} d_{HKS}(u_s,u_{s+h})\leq \sqrt{\entropyproduction(u_s)}.\end{equation} Consequently, \begin{multline} \frac d {ds}\Big|^+ d_{HKS}(u_t,u_s)=\limsup_{h\to 0} \frac {d_{HKS}(u_t,u_{s+h})-d_{HKS}(u_t,u_s)}{h}\\ \leq \limsup_{h\to 0} \frac {d_{HKS}(u_s,u_{s+h})}{h} \leq  \sqrt{\entropyproduction(u_s)},\ t\leq s.\label{eq:eep-2}\end{multline} 

Consider the function $$\phi(s):=2\sqrt{C_1 \entropy(u_s)}+d_{HKS}(u_t,u_s), \ s\geq t.$$ By \eqref{eq:eep-df}, \eqref{eq:eep-1} and \eqref{eq:eep-2}, $$\frac d {ds}\Big|^+\phi(s)\le \left[-\sqrt{\frac {C_1\entropyproduction(u_s)}{\entropy(u_s)}}+1\right]\sqrt{\entropyproduction(u_s)}\leq 0.$$ Therefore \begin{equation}\label{eq:cauchy} d_{HKS}(u_t,u_s)\leq \phi(s)\le \phi(t)=2\sqrt{C_1 \entropy(u_t)}\leq 2 \sqrt{C_1 e^{-\gamma t}\entropy(u^0)}.\end{equation} The cone $(\Mm, d_{HK})$ is a complete metric space (cf. \cite{KMV16A}), hence \cite{BH99} the sphere $(\Pp, d_{HKS})$ is also complete. Now \eqref{eq:cauchy} yields existence of $u_\infty\in \Pp$ such that $u_t\to u_\infty$ as $t\to \infty$ in $(\Pp, d_{HKS})$ and thus weakly as probability measures. Fix $c>0$ such that there exists $m_{-c}$ (actually any $c>0$ would work since the second alternative in \eqref{eq:large-u} is assumed). Observing that $E_{u}=-f>c$ for $u>m_{-c}(x)$ we can deduce existence of a continuous function $a:\mathbb T^d \to \R$ such that \begin{equation} E(x,u)>a(x)+cu. \label{Ebound} \end{equation} Taking into account that $E_{uu}>0$ and using the results of \cite[Subsection 6.4.5]{FL07} we infer that the entropy functional $\entropy$ is lower-semicontinuous w.r.t. the weak convergence , whence $\entropy(u_\infty)=0$, and $u_\infty=m$. Letting $t=0$ and $s\to +\infty$ in \eqref{eq:cauchy}, we get the claim \eqref{eq:tala}. \end{proof}

\begin{corollary}[Stability of the spherical gradient flows w.r.t. the spherical distance]
In Theorem~\ref{th:convergence} with $\Omega = \mathbb T^d$ one has
\begin{equation}
d_{HKS}^2(u_t,m) \leq C\,
\entropy(u^0) \mathrm{e}^{-\gamma t} \quad \text{a.~a.\ } t > 0
.
\end{equation}
\end{corollary}

Using a similar argument and the entropy-entropy production inequality 
obtained in \cite[Theorem 2.9]{KV17}  for the Hellinger-Kantorovich 
gradient flows, we can get  a transportation inequality for the conic 
distance.  From now on we do not assume that the implicit function 
$m$ defined by \eqref{impm} has mass $1$ (cf. Remark \ref{rem:en}).
\begin{theorem}[Conic Talagrand inequality]
\label{th:ineq-tala2}
Suppose that~$f$ satisfies~\eqref{eq:f1a}--\eqref{eq:f5}, 
\eqref{impm}.  Let $u^0\in L^1(\mathbb T^d)$, $\entropy(u^0)<\infty$. 
Then \begin{equation}
\label{eq:tala2}
d_{HK}^2(u^0,m) \leq C\, \entropy(u^0),
\end{equation}
with $C$ independent of $u^0$.
\end{theorem}

\begin{proof}   As in the previous proof, we may assume that $u^0$ is smooth and strictly positive. In the case when $\entropy(u^0)<\entropy(0)$ the proof mimicks the previous one, basically substituting the objects related to the spherical Hellinger-Kantorovich distance with the conic ones. Let us merely describe the small differences that show up. Consider the classical solution $u$ to the conic Hellinger-Kantorovich gradient flow \cite{KV17}. The condition \eqref{eq:large-u} is not needed because the conic entropy-entropy production  inequality \cite[Theorem 2.9]{KV17} does not require it.  However, in order to apply that theorem we need to find a set $U$ containing the trajectory $u_t$ of the conic gradient flow starting from $u^0$ such that no sequence in $U$ converges to $0$ in the sense of measures. An argument involving Lebesgue's Dominated Convergence Theorem shows that we can simply take \begin{equation*}
U = \left\{ u \in L^\infty_+(\Omega):\ \entropy(u) \le \entropy(u^0)<\entropy(0)\right\}. 
\end{equation*} It remains to treat the case $\entropy(u^0)\ge \entropy(0)$. Since $\entropy(0)$ is a positive constant, it suffices to prove the inequality \begin{equation}
\label{eq:tala3}
d_{HK}^2(u^0,m) \leq C\, (1+\entropy(u^0)).
\end{equation} We recall \cite{CP18,LM17,BV18} the upper bound for the Hellinger-Kantorovich distance in terms of the masses, $$d_{HK}^2(u^0,m)\leq 4\left(\int_{\mathbb T^d} u^0 + \int_{\mathbb T^d} m\right).$$ Consequently, it is enough to show \begin{equation}
\label{eq:tala4}
\int_{\mathbb T^d} u^0 \leq C\, (1+\entropy(u^0)).
\end{equation} Let $c$ be a small positive constant such that the implicit function $m_{-c}$ exists.
As in the previous proof, we can deduce \eqref{Ebound} with $c$ just defined and some function $a(x)$ independent of $u$. Hence,
\begin{equation*}
\int_{\mathbb T^d}  u^0 \leq C+c^{-1}\int_{\mathbb T^d}  E(x,u^0(x))\cd x,
\end{equation*} 
proving \eqref{eq:tala4}. \end{proof}

\begin{corollary}[Stability of the conic gradient flows w.r.t. the conic distance]
In \cite[Theorem~1.12]{KV17} with $\Omega = \mathbb T^d$ one has
\begin{equation}
d_{HK}^2(u_t,m) \leq C\,
\entropy(u^0) \mathrm{e}^{-\gamma t} \quad \text{a.~a.\ } t > 0
.
\end{equation}
\end{corollary}

\begin{remark} \label{r:cons} Inequality \eqref{eq:tala2} follows from 
\eqref{eq:tala} and \eqref{compa} provided $u^0$ and $m$ are 
probability measures. However, when the masses of $u^0$ and $m$ do not 
coincide,  \eqref{eq:tala2} is not an immediate consequence of 
\eqref{eq:tala}.  Hence, the conic inequality is new even for the displacement 
convex entropies.\end{remark}

\subsection{Examples}
\label{ss:examples}

Let us consider three simple examples of~$f$:
\begin{align}
f_1(x, u) & = \frac{1 - u^\alpha}{\alpha},
\\
f_2(x, u) &= - \log u - V(x),
\\
f_3(x, u) &= - \log \frac{u}{\sqrt{1 + u^2}} - \frac 12 \log 2
,
\end{align}
where $\alpha \ne 0$ is a real parameter, and $V \in C^2(\overline \Omega)$ 
satisfies
\begin{equation*}
\int_\Omega e^{-V(x)} \, dx = 1
.
\end{equation*}
Note that $V$ does not need to be convex. For the sake of simplicity we assume that $|\Omega|=1$.

The corresponding primitives are
\begin{align}
E_1
& =
\begin{cases}
\frac 1{\alpha(\alpha + 1)}\left(  u^{\alpha + 1}-(\alpha + 
1)u + \alpha \right), & \text{if } \alpha \ne -1
\\
u - \log u - 1, & \text{if } \alpha = -1,
\end{cases}
\\
E_2 &
= u \log u - u + 1 + uV(x),
\\
E_3 &
= u \log \frac{u}{\sqrt{1 + u^2}} - \arctan u + \frac 12 \left(u\log 2 
+ \frac \pi2 \right)
.
\end{align}
In the Otto--Wasserstein setting, $E_1$ corresponds to the porous medium flow, 
$E_2$ corresponds to the linear Fokker--Planck equation, and $E_3$ corresponds 
to the arctangential heat flow~\cite{Bre19}.

Theorems~\ref{th:ineq-main} and~\ref{th:convergencew} are applicable without 
any further assumptions.  Theorem~\ref{th:ex}, Remark~\ref{rem:linfty}, and 
Theorem~\ref{th:ineq-tala2} work in all the cases except for $E_1$ with 
$\alpha \le -1$. Theorems~\ref{th:eep}, \ref{th:convergence}, 
and~\ref{th:ineq-tala} can be applied to $E_1$ with $\alpha > 0$ and~$E_2$.

The entropy associated with~$E_1$ is non-strictly geodesically convex in the 
Wasserstein space provided $\alpha \ge - 1/d$.  The convexity in the the 
Hellinger-Kantorovich spaces can be secured if $\alpha > 0$ (only in the conic 
case) or $d = 2$, $\alpha = -1/2$ or $d = 1$, $\alpha \in [-2/3, -1/2]$.  If 
$V$ is $\lambda$-convex, the entropy associated with~$E_2$ is geodesically 
$\lambda$-convex with the same $\lambda$ in the Wasserstein space but is not 
even semi-convex in the Hellinger-Kantorovich spaces.  Finally, $E_3$ is 
semi-convex neither in the Wasserstein space nor in the Hellinger-Kantorovich 
spaces.  The convexity can be checked formally via the positive definiteness 
of the Hessians (in the sense of the Otto calculus) of the corresponding 
entropies.  We refer to~\cite{Vil03} for the Wasserstein case 
and~\cite{KMV16A} for the conic Hellinger-Kantorovich case.

\subsection*{Acknowledgment}
The second author would like to thank J.A. Carrillo and W. Shi for helpful discussions. The research was partially supported by the Portuguese Government through FCT/MCTES and by the ERDF through PT2020 (projects UID/MAT/00324/2019, PTDC/MAT-PUR/28686/2017 and TUBITAK/0005/2014).

\subsection*{Conflict of interest statement} We have no conflict of interest to declare.

\bibliographystyle{abbrv}
\bibliography{biblios} 

\begin{thebibliography}{10}

\bibitem{alr16}
H.~Alrachid, T.~Leli{\`e}vre, and R.~Talhouk.
\newblock Local and global solution for a nonlocal {F}okker--{P}lanck equation
  related to the adaptive biasing force process.
\newblock {\em Journal of Differential Equations}, 260(9):7032--7058, 2016.

\bibitem{Am93}
H.~Amann.
\newblock Nonhomogeneous linear and quasilinear elliptic and parabolic boundary
  value problems.
\newblock In {\em Function spaces, differential operators and nonlinear
  analysis}, pages 9--126. Springer, 1993.

\bibitem{AGS06}
L.~Ambrosio, N.~Gigli, and G.~Savar{\'e}.
\newblock {\em Gradient Flows: in Metric Spaces and in the Space of Probability
  Measures}.
\newblock Basel: Birkh{\"a}user Basel, 2008.

\bibitem{BE85}
D.~Bakry and M.~{\'E}mery.
\newblock Diffusions hypercontractives.
\newblock In {\em S{\'e}minaire de Probabilit{\'e}s XIX 1983/84}, pages
  177--206. Springer, 1985.

\bibitem{Bre19}
Y.~Brenier.
\newblock Geometric origin and some properties of the arctangential heat
  equation.
\newblock {\em Tunis. J. Math.}, 1(4):561--584, 2019.

\bibitem{BV18}
Y.~{Brenier} and D.~{Vorotnikov}.
\newblock {On optimal transport of matrix-valued measures}.
\newblock {\em ArXiv e-prints}, Aug. 2018.

\bibitem{BH99}
M.~R. Bridson and A.~Haefliger.
\newblock {\em Metric spaces of non-positive curvature}, volume 319 of {\em
  Grundlehren der Mathematischen Wissenschaften [Fundamental Principles of
  Mathematical Sciences]}.
\newblock Springer-Verlag, Berlin, 1999.

\bibitem{Bur}
D.~Burago, Y.~Burago, and S.~Ivanov.
\newblock {\em A course in metric geometry}.
\newblock AMS, 2001.

\bibitem{CJM01}
J.~Carrillo, A.~J{\"u}ngel, P.~Markowich, G.~Toscani, and A.~Unterreiter.
\newblock Entropy dissipation methods for degenerate parabolic problems and
  generalized {S}obolev inequalities.
\newblock {\em Monatshefte f{\"u}r Mathematik}, 133(1):1--82, 2001.

\bibitem{CMcV03}
J.~A. Carrillo, R.~J. McCann, and C.~Villani.
\newblock Kinetic equilibration rates for granular media and related equations:
  entropy dissipation and mass transportation estimates.
\newblock {\em Revista Matematica Iberoamericana}, 19(3):971--1018, 2003.

\bibitem{CP18}
L.~Chizat, G.~Peyr{\'e}, B.~Schmitzer, and F.-X. Vialard.
\newblock An interpolating distance between optimal transport and
  {F}isher--{R}ao metrics.
\newblock {\em Foundations of Computational Mathematics}, 18(1):1--44, 2018.

\bibitem{CPSV18}
L.~Chizat, G.~Peyr{\'e}, B.~Schmitzer, and F.-X. Vialard.
\newblock Unbalanced optimal transport: Dynamic and {K}antorovich formulations.
\newblock {\em Journal of Functional Analysis}, 274(11):3090--3123, 2018.

\bibitem{CGH04}
D.~Cordero-Erausquin, W.~Gangbo, and C.~Houdr\'e.
\newblock Inequalities for generalized entropy and optimal transportation.
\newblock In {\em Recent advances in the theory and applications of mass
  transport}, volume 353 of {\em Contemp. Math.}, pages 73--94. Amer. Math.
  Soc., Providence, RI, 2004.

\bibitem{cosner05}
C.~Cosner.
\newblock A dynamic model for the ideal-free distribution as a partial
  differential equation.
\newblock {\em Theoretical Population Biology}, 67(2):101--108, 2005.

\bibitem{cos13}
C.~Cosner.
\newblock Beyond diffusion: conditional dispersal in ecological models.
\newblock In J.~{M}allet-{P}aret~et al., editor, {\em Infinite Dimensional
  Dynamical Systems}, pages 305--317. Springer, 2013.

\bibitem{CW13}
C.~Cosner and M.~Winkler.
\newblock Well-posedness and qualitative properties of a dynamical model for
  the ideal free distribution.
\newblock {\em Journal of mathematical biology}, 69(6-7):1343--1382, 2014.

\bibitem{KHL17}
P.~El~Kettani, D.~Hilhorst, and K.~Lee.
\newblock A well-posedness result for a mass conserved {A}llen-{C}ahn equation
  with nonlinear diffusion.
\newblock In {\em Proceedings of Equadiff 2017 Conference}, pages 201--210,
  2017.

\bibitem{E13}
P.~Embrechts and M.~Hofert.
\newblock A note on generalized inverses.
\newblock {\em Mathematical Methods of Operations Research}, 77(3):423--432,
  2013.

\bibitem{FKS82}
E.~B. Fabes, C.~E. Kenig, and R.~P. Serapioni.
\newblock The local regularity of solutions of degenerate elliptic equations.
\newblock {\em Comm. Partial Differential Equations}, 7(1):77--116, 1982.

\bibitem{FL07}
I.~Fonseca and G.~Leoni.
\newblock {\em Modern Methods in the Calculus of Variations: $L^p$ Spaces}.
\newblock Springer Science \& Business Media, 2007.

\bibitem{F04}
T.~D. Frank.
\newblock Asymptotic properties of nonlinear diffusion, nonlinear
  drift-diffusion, and nonlinear reaction-diffusion equations.
\newblock {\em Ann. Phys.}, 13(7-8):461--469, 2004.

\bibitem{GLM17}
T.~Gallou{\"e}t, M.~Laborde, and L.~Monsaingeon.
\newblock An unbalanced optimal transport splitting scheme for general
  advection-reaction-diffusion problems.
\newblock {\em arXiv:1704.04541}, 2017.

\bibitem{MG16}
T.~O. Gallou\"et and L.~Monsaingeon.
\newblock A {JKO} splitting scheme for {K}antorovich-{F}isher-{R}ao gradient
  flows.
\newblock {\em SIAM J. Math. Anal.}, 49(2):1100--1130, 2017.

\bibitem{T18}
I.~T. Heilmann, U.~H. Thygesen, and M.~P. S{\o}rensen.
\newblock Spatio-temporal pattern formation in predator-prey systems with
  fitness taxis.
\newblock {\em Ecological Complexity}, 34:44--57, 2018.

\bibitem{HMTW16}
D.~Hilhorst, H.~Matano, T.~N. Nguyen, and H.~Weber.
\newblock On the large time behavior of the solutions of a nonlocal ordinary
  differential equation with mass conservation.
\newblock {\em Journal of Dynamics and Differential Equations},
  28(3-4):707--731, 2016.

\bibitem{HY95}
B.~Hu and H.-M. Yin.
\newblock Semilinear parabolic equations with prescribed energy.
\newblock {\em Rendiconti del Circolo Matematico di Palermo}, 44(3):479--505,
  1995.

\bibitem{iak16}
M.~{Iacobelli}.
\newblock {Asymptotic analysis for a very fast diffusion equation arising from
  the 1D quantization problem}.
\newblock {\em DCDS-A}.
\newblock To appear.

\bibitem{IS18}
M.~{Iacobelli}, F.~{Patacchini}, and F.~{Santambrogio}.
\newblock {Weighted ultrafast diffusion equations: from well-posedness to
  long-time behaviour}.
\newblock {\em Archive for Rational Mechanics and Analysis}.
\newblock To appear.

\bibitem{JKO}
R.~Jordan, D.~Kinderlehrer, and F.~Otto.
\newblock The variational formulation of the {F}okker--{P}lanck equation.
\newblock {\em SIAM journal on mathematical analysis}, 29(1):1--17, 1998.

\bibitem{KMV16A}
S.~Kondratyev, L.~Monsaingeon, and D.~Vorotnikov.
\newblock A new optimal transport distance on the space of finite {R}adon
  measures.
\newblock {\em Adv. Differential Equations}, 21(11-12):1117--1164, 2016.

\bibitem{KV19A}
S.~{Kondratyev} and D.~{Vorotnikov}.
\newblock {Convex Sobolev inequalities related to unbalanced optimal
  transport}.
\newblock In prep.

\bibitem{KV17}
S.~Kondratyev and D.~Vorotnikov.
\newblock Nonlinear {F}okker-{P}lanck equations with reaction as gradient flows
  of the free energy.
\newblock {\em arXiv preprint arXiv:1706.08957}, 2017.

\bibitem{LM17}
V.~{Laschos} and A.~{Mielke}.
\newblock {Geometric properties of cones with applications on the
  Hellinger-Kantorovich space, and a new distance on the space of probability
  measures}.
\newblock {\em ArXiv e-prints}, Dec. 2017.

\bibitem{LMS16}
M.~Liero, A.~Mielke, and G.~Savar\'e.
\newblock Optimal transport in competition with reaction: the
  {H}ellinger-{K}antorovich distance and geodesic curves.
\newblock {\em SIAM J. Math. Anal.}, 48(4):2869--2911, 2016.

\bibitem{LMS18}
M.~Liero, A.~Mielke, and G.~Savar{\'e}.
\newblock Optimal entropy-transport problems and a new
  {H}ellinger--{K}antorovich distance between positive measures.
\newblock {\em Inventiones mathematicae}, 211(3):969--1117, 2018.

\bibitem{mc90}
A.~D. MacCall.
\newblock {\em Dynamic geography of marine fish populations}.
\newblock Washington Sea Grant Program Seattle, 1990.

\bibitem{Mag12}
F.~Maggi.
\newblock {\em Sets of Finite Perimeter and Geometric Variational Problems: An
  Introduction to Geometric Measure Theory}.
\newblock Cambridge Studies in Advanced Mathematics. Cambridge University
  Press, 2012.

\bibitem{Mazja}
V.~G. Maz'ja.
\newblock {\em Sobolev spaces}.
\newblock Springer Series in Soviet Mathematics. Springer-Verlag, Berlin, 1985.
\newblock Translated from the Russian by T. O. Shaposhnikova.

\bibitem{otto01}
F.~Otto.
\newblock The geometry of dissipative evolution equations: the porous medium
  equation.
\newblock {\em Comm. Partial Differential Equations}, 26(1-2):101--174, 2001.

\bibitem{OV00}
F.~Otto and C.~Villani.
\newblock Generalization of an inequality by {T}alagrand and links with the
  logarithmic sobolev inequality.
\newblock {\em Journal of Functional Analysis}, 173(2):361--400, 2000.

\bibitem{QS07}
P.~Quittner and P.~Souplet.
\newblock {\em Superlinear parabolic problems: blow-up, global existence and
  steady states}.
\newblock Springer Science \& Business Media, 2007.

\bibitem{Rez15}
F.~Rezakhanlou.
\newblock Optimal transport problem and contact structures.
\newblock {\em preprint}, 2015.

\bibitem{Rou}
T.~Roub{\'\i}{\v{c}}ek.
\newblock {\em Nonlinear partial differential equations with applications},
  volume 153.
\newblock Springer Science \& Business Media, 2013.

\bibitem{RS92}
J.~Rubinstein and P.~Sternberg.
\newblock Nonlocal reaction—diffusion equations and nucleation.
\newblock {\em IMA Journal of Applied Mathematics}, 48(3):249--264, 1992.

\bibitem{Sou98}
P.~Souplet.
\newblock Blow-up in nonlocal reaction-diffusion equations.
\newblock {\em SIAM Journal on Mathematical Analysis}, 29(6):1301--1334, 1998.

\bibitem{St93}
E.~M. Stein.
\newblock {\em Harmonic Analysis: Real-Variable Methods, Orthogonality, and
  Oscillatory Integrals}.
\newblock Princeton University Press, 1993.

\bibitem{T96}
M.~Talagrand.
\newblock Transportation cost for {G}aussian and other product measures.
\newblock {\em Geometric \& Functional Analysis GAFA}, 6(3):587--600, 1996.

\bibitem{Vil03}
C.~Villani.
\newblock {\em Topics in optimal transportation}.
\newblock American Mathematical Soc., 2003.

\bibitem{Vil08}
C.~Villani.
\newblock {\em Optimal transport: old and new}.
\newblock Springer Science \& Business Media, 2008.

\end{thebibliography}
\end{document}